\documentclass{amsart}

\usepackage{calc}
\usepackage{tikz}
\usepackage{pgfplots}
\usetikzlibrary{shapes}
\usepackage{amsmath,amssymb}
 \usepackage[foot]{amsaddr}
\let\oldtocsection=\tocsection
 
\let\oldtocsubsection=\tocsubsection 
 
\let\oldtocsubsubsection=\tocsubsubsection
 
\renewcommand{\tocsection}[2]{\vspace{0.5em}\hspace{0em}\oldtocsection{#1}{#2}}
\renewcommand{\tocsubsection}[2]{\vspace{0.5em}\hspace{1em}\oldtocsubsection{#1}{#2}}
\renewcommand{\tocsubsubsection}[2]{\vspace{0.5em}\hspace{2em}\oldtocsubsubsection{#1}{#2}}
\usepackage{graphicx,cancel,xcolor,hyperref,comment,graphicx,geometry}
 
\let\originallesssim\lesssim
\DeclareRobustCommand{\lesssim}{%
  \mathrel{\mathpalette\lowersim\originallesssim}%
}
\usepackage{tikz}
\usepackage{graphicx,url,etoolbox}
\usepackage{lipsum}
\usepackage{dsfont}
\usepackage{subfigure}
\usetikzlibrary{hobby,patterns}

\usepackage{fancyhdr}
\usepackage{lastpage}

\newtheorem{theoreme}{Theorem}[section]
\newtheorem{pro}[theoreme]{Proposition}
\newtheorem{lemma}[theoreme]{Lemma}

\theoremstyle{definition}

\numberwithin{equation}{section}

 \renewenvironment{proof}{{\bfseries \noindent Proof.}}{\demo}
\newcommand\xqed[1]{%
  \leavevmode\unskip\penalty9999 \hbox{}\nobreak\hfill
  \quad\hbox{#1}}
\newcommand\demo{\xqed{$\square$}}

\hypersetup{bookmarks, colorlinks, urlcolor=blue, citecolor=red, linkcolor=blue, hyperfigures, pagebackref,
    pdfcreator=LaTeX, breaklinks=true, pdfpagelayout=SinglePage, bookmarksopen=true,bookmarksopenlevel=2}

\usepackage{comment}
\def\u2{\u^2}
\def\u3{\u^3}
\def\u4{\u^4}
\def\u5{\u^5}
\def\y1{\y^1}
\def\y2{\y^2}
\def\y3{\y^3}
\def\y4{\y^4}
\def\y5{\y^5}

\def\R{\mathbb R}

\def\la {{\lambda}}

\newcommand {\nc}   {\newcommand}
\nc {\be}   {\begin{equation}} \nc {\ee}   {\end{equation}} \nc
{\beq}  {\begin{eqnarray}} \nc {\eeq}  {\end{eqnarray}} \nc {\beqs}
{\begin{eqnarray*}} \nc {\eeqs} {\end{eqnarray*}}
\def\edc{\end{document}}

\providecommand{\abs}[1]{\lvert#1\rvert}

\DeclareMathOperator{\supp}{supp}
\usepackage{tikz}
\usetikzlibrary{decorations.pathmorphing,patterns,scopes,intersections,calc}
\usepackage{caption}

   \usepackage{soul}

\begin{document}
\title[\fontsize{7}{9}\selectfont  ]{Polynomial energy decay rate of a 2D Piezoelectric beam with magnetic effect on a rectangular domain without geometric conditions}
\author{Mohammad Akil$^{1}$ and Virginie R\'egnier$^{1}$  \vspace{0.5cm}\\
$^1$Univ. Polytechnique  Hauts-de-France, INSA Hauts-de-France,  CERAMATHS-Laboratoire de Mat\'eriaux C\'eramiques et de Math\'ematiques, F-59313 Valenciennes, France\\ \\ 
Email: mohammad.akil@uphf.fr, virginie.regnier@uphf.fr}

\setcounter{equation}{0}
\begin{abstract}
In this paper, we investigate the stability of coupled equations modelling a 2D piezoelectric beam with magnetic effect with only one local viscous damping on a rectangular domain without geometric conditions. We prove that the energy of the system decays polynomially with the rate $t^{-1}$.\\[0.1in]
\end{abstract}

\maketitle
\pagenumbering{roman}
\maketitle
\tableofcontents
\pagenumbering{arabic}
\setcounter{page}{1}
\vspace{-1.5cm}\section{Introduction} 

It is known, since the 19th century, that materials such as quartz, Rochelle salt and barium titanate under pressure produce electric charge\slash voltage: this phenomenon is called the direct piezoelectric effect and was discovered by brothers Pierre and Jacques Curie in 1880. These same materials, when subjected to an electric field, produce proportional geometric tension. Such a  phenomenon is known as the converse piezoelectric effect and was discovered by Gabriel Lippmann in 1881.\\
Morris and Ozer proposed a piezoelectric beam model with a magnetic effect, based on the Euler-Bernoulli and Rayleigh  beam theory for small displacement (the same equations for the model are obtained if Midlin-Timoshenko small displacement assumptions are used). They considered an elastic beam covered by a piezoelectric material on its upper and lower surfaces, isolated by the edges and connected to an external electrical circuit to feed charge to the electrodes. As the voltage is prescribed at the electrodes, the following Lagrangian is considered 
\begin{equation}\label{Lagrangian}
	\mathcal{L}=\int_0^T\left[\bold{K-(P+E)+B+W}\right]dt,
\end{equation}
where $\bold{K}$, $\bold{P+E}$, $\bold{B}$ and $\bold{W}$ represent the (mechanical) kinetic energy, total stored energy, magnetic energy (electrical kinetic) of the beam and the work done by external forces, respectively, for a beam with length $L$ and thickness $h$ and considering $v=v(x,t)$, $w=w(x,t)$ and $p=p(x,t)$ as functions that represent the longitudinal displacement of the center line, transverse displacement of the beam and the total load of the electric displacement along the transverse direction at each point $x$, respectively. So, one can assume that 
\begin{equation}\label{newL}
	\begin{array}{cc}
	\displaystyle
	\bold{P+E}=\frac{h}{2}\int_0^L\left[\alpha\left(v_x^2+\frac{h^2}{12}w_{xx}^2-2\gamma\beta v_xp_x+\beta p_x^2\right)\right]dx, & \displaystyle
	\bold{B}=\frac{\mu h}{2}\int_0^Lp_t^2dx,\\[0.1in]
	\displaystyle
	\bold{K}=\frac{\rho h}{2}\int_0^L\left[v_t^2+\left(\frac{h^2}{12}+1\right)\omega_t^2 \right] dx,& \displaystyle 
	\bold{W}=-\int_0^Lp_xV(t)dx,
	\end{array}
\end{equation}
where $V(t)$ is the voltage applied at the electrode. From Hamilton's principle for admissible displacement variations $\left\{v,w,p\right\}$ of $\mathcal{L}$  and observing that the only external force acting on the beam is the voltage at the electrodes (the bending equation is decoupled, see \cite{Morris-Ozer2013,Morris-Ozer2014}), they got the system 
\begin{equation}\label{piezo}
\begin{array}{c}
\rho v_{tt}-\alpha v_{xx}+\gamma \beta p_{xx}=0,\\
\mu p_{tt}-\beta p_{xx}+\gamma \beta v_{xx}=0,
\end{array}	
\end{equation}
where $\rho, \alpha, \gamma, \mu$ and $\beta$ denote the mass density, elastic stiffness, piezoelectric coefficient, magnetic permeability, water resistance coefficient of the beam and the prescribed voltage on electrodes of beam, respectively, and in addition, the relationship 
\begin{equation}\label{alpha1}
\alpha=\alpha_1+\gamma^2\beta.	
\end{equation}
They assumed that the beam is fixed at $x=0$ and free at $x=L$, and thus they got (from modelling) the following boundary conditions  
\begin{equation}\label{bc}
\begin{array}{c}
v(0,t)=\alpha v_x(L,t)-\gamma \beta p_x(L,t)=0,\\
p(0,t)=\beta p_x(L,t)-\gamma \beta v_x(L,t)=- \displaystyle{\frac{V(t)}{h}}.	
\end{array}	
\end{equation}
Then, the authors considered $V(t)=kp_t(L,t)$ (electrical feedback controller) in \eqref{bc} and established strong stabilization for almost all system parameters and exponential stability for system parameters in a null measure set. In \cite{Ramos2018} Ramos et al. inserted a dissipative term $\delta v_t$ in the first equation of \eqref{piezo} , where $\alpha>0$ is a constant and considered the following boundary conditions 
\begin{equation}\label{Ramos-bc}
\begin{array}{c}
v(0,t)=\alpha v_x(L,t)-\gamma \beta p_x(L,t)=0,\\
p(0,t)=\beta p_x(L,t)-\gamma \beta v_x(L,t)=0.	
\end{array}	
\end{equation}
The authors showed, by using energy method, that the system's energy decays exponentially. This means that the friction term and the magnetic effect work together in order to uniformly stabilize the system. In \cite{Abdelaziz2}, the authors considered a one-dimensional piezoelectric beam with magnetic effect damped with a weakly nonlinear feedback in the presence of a nonlinear delay term.They established an energy decay rate under appropriate assumptions on the weight of the delay. In \cite{AnLiuKong}, the authors studied the stability of piezoelectric beams with magnetic effects of fractional derivative type and with/ without thermal effects of Fourier's law. They obtained an exponential stability by taking two boundary fractional dampings and additional thermal effects. 
In \cite{Zhang2022}, the authors studied the longtime behavior of a kind of fully magnetic effected nonlinear multi-dimensional piezoelectric beams with viscoelastic infinite memory. An exponential decay of the solution to the nonlinear coupled PDE's system is established by the energy estimation method under certain conditions.
\noindent In \cite{https://doi.org/10.48550/arxiv.2204.00283}, the author investigates the stabilization of a system of piezoelectric beams under (Coleman or Pipkin)-Gurtin thermal law with magnetic effect. First, he study the Piezoelectric-Coleman-Gurtin system and he obtain an exponential stability result. Next, he consider the Piezoelectric-Gurtin-Pipkin system and he establish a polynomial energy decay rate of type $t^{-1}$. In \cite{Abdelaziz1}, the authors considered a one-dimensional dissipative system of piezoelectric beams with magnetic effect and localized viscous damping. They proved that the system is exponentially stable using a damping mechanism acting only on one component and on a small part of the beam. \\
\noindent What happens if we go to dimension 2? For this aim, let $\Omega=\square=(0,1)^2$ be the unit square  and $\partial\Omega=\Gamma_0\cup \Gamma_1$ (See Figure \ref{SQUARE}), such that 
 $$
\Gamma_{0,1}=\left(\left\{0\right\}\times (0,1)\right)\cup \left((0,1)\times \left\{0\right\}\right)\quad \text{and}\quad \Gamma_{1,1}=\left(\left\{1\right\}\times (0,1)\right)\cup ((0,1)\times \left\{1\right\}).
$$
\noindent In the present work, we consider the fully dynamic magnetic effects in a model for a piezoelectric beam with only one viscous damping on a rectangular domain, whose dynamic behaviour is described by an elasticity equation and a charge equation coupled via the piezoelectric constants, which is given as follows 
\begin{equation}\label{1PIEZO-2D}
\left\{\begin{array}{l}
\rho v_{tt}(X,t)-\alpha \Delta v(X,t)+\gamma \beta \Delta p(X,t)+ d(\cdot) v_t(X,t)=0, \quad X=(x,y)\in \Omega,\ t>0,\\[0.1in]
\mu p_{tt}(X,t)-\beta \Delta p(X,t)+\gamma\beta \Delta v(X,t)=0,\quad X=(x,y)\in \Omega,\ t>0,
\end{array}
\right.
\end{equation}
where $v(X,t)$ and $p(X,t)$ respectively denote the transverse displacement of the beam and the total load of the electric displacement along the transverse direction at each point $x\in \Omega$. The constant coefficients $\rho, \alpha, \gamma, \mu, \beta>0$ are the mass density per unit volume, elastic stiffness, piezoelectric coefficient, magnetic permeability, impermeability coefficient of the beam, respectively and satisfy $\alpha>\gamma^2\beta$. System \eqref{1PIEZO-2D} is subjected to the following boundary conditions 
\begin{equation*}\label{2PIEZO-2D}
\left\{\begin{array}{l}
v(X,t)=p(X,t)=0,\quad X=(x,y)\in \Gamma_0,\ t>0,\\[0.1in]
\displaystyle 
\alpha \frac{\partial v}{\partial \nu}(X,t)-\gamma\beta\frac{\partial p}{\partial \nu}(X,t)=\beta \frac{\partial p}{\partial \nu}(X,t)-\gamma \beta\frac{\partial v}{\partial \nu}(X,t)=0,\quad X=(x,y)\in \Gamma_1,\ t>0,	
\end{array}
\right.	
\end{equation*}
Using the fact that $\alpha=\alpha_1+\gamma^2\beta$ (see \eqref{alpha1}), then the above boundary conditions can be replaced by 
\begin{equation}\label{2PIEZO-2D}
\left\{\begin{array}{l}
v(X,t)=p(X,t)=0,\quad X=(x,y)\in \Gamma_0,\ t>0,\\[0.1in]
\displaystyle 
\frac{\partial v}{\partial \nu}(X,t)=\frac{\partial p}{\partial \nu}(X,t)=0,\quad X=(x,y)\in \Gamma_1,\ t>0.
\end{array}
\right.	
\end{equation}
System \eqref{1PIEZO-2D} is considered with the following initial data
\begin{equation}\label{Piezo-Initial}
v(X,0)=v_0(X),\quad v_t(X,0)=v_1(X),\quad p(X,0)=p_0(X),\quad p_t(X,0)=p_1(X),\quad X=(x,y)\in \Omega. 	
\end{equation}

\noindent Let $d\in L^{\infty}(0,1)$,  depending only on $x$ such that
\begin{equation}\label{dsquare}
d(x)\geq d_0>0\ \text{in}\ (a,b)\quad \text{and}\quad d(x)=0\ \text{in}\ (0,1)\backslash (a,b).	
\end{equation}
We set $\omega_d=(\supp d)^{\circ}\times (0,1)$ (See Figure \ref{SQUARE}).\\

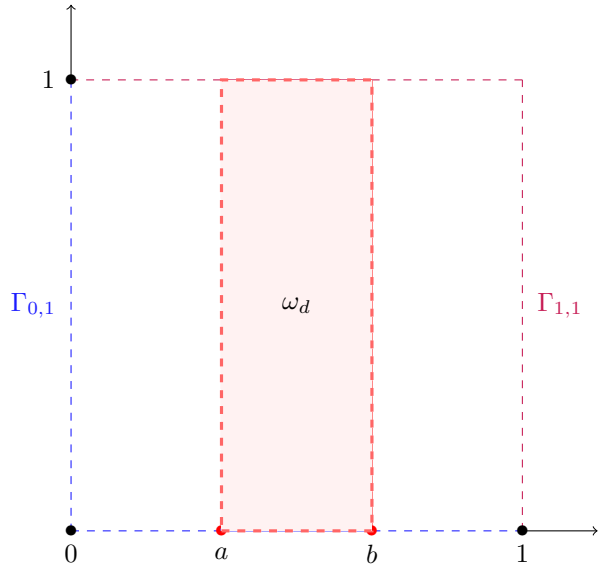
\begin{figure}[h!]
\begin{tikzpicture}
\draw[dashed,color=blue!90](0,0)--(6,0);
\draw[dashed,color=purple!90](6,0)--(6,6);
\draw[dashed,color=purple!90](6,6)--(0,6);
\draw[dashed,color=blue!90](0,6)--(0,0);
\draw[dashed,red](2,0)--(4,0);
\draw[dashed,red](4,0)--(4,6);
\draw[dashed,red](4,6)--(2,6);
\draw[dashed,red](2,6)--(2,0);
\node[blue!90] at (-0.5,3) {\scalebox{1}{$\Gamma_{0,1}$}};
\node[purple!90] at (6.5,3) {\scalebox{1}{$\Gamma_{1,1}$}};
\node[red] at (2,0) {\scalebox{1}{$\bullet$}};
\node[black] at (2,-0.3) {\scalebox{1}{$a$}};
\node[red] at (4,0) {\scalebox{1}{$\bullet$}};
\node[black] at (4,-0.3) {\scalebox{1}{$b$}};
\node[black] at (6,0) {\scalebox{1}{$\bullet$}};
\node[black] at (6,-0.3) {\scalebox{1}{$1$}};
\node[black] at (0,0) {\scalebox{1}{$\bullet$}};
\node[black] at (0,-0.3) {\scalebox{1}{$0$}};
\node[black] at (0,6) {\scalebox{1}{$\bullet$}};
\node[black] at (-0.3,6) {\scalebox{1}{$1$}};
\draw[color=blue!60, fill=blue!5, very thick,dashed] (2,0)--(4,6);	
\filldraw[color=red!60, fill=red!5, very thick,dashed] (2,0) rectangle (4,6);
\node[black] at (3,3) {\scalebox{1}{$\omega_d$}};
\draw[arrows=->] (6,0) -- (7,0);
\draw[arrows=->] (0,6) -- (0,7); 
\end{tikzpicture}
\caption{Model Describing $\Omega$.}\label{SQUARE}	
\end{figure}
\noindent There exist a few results concerning wave equations or coupled wave equations on a rectangular or cylindrical domain  with different kinds of damping \cite{RaoLiu01,Stahn2017,BPS,doi:10.1137/20M1332499, Hayek, Akil2022, https://doi.org/10.48550/arxiv.2111.14554}. But to the best of our knowledge, it seems that no result in the literature exists concerning the case of a 2D or multidimensional Piezoelectric beam  with local damping. \\
\\
\noindent This paper is organized as follows: the second section is devoted to the case of a rectangular domain.  The well-posedness of the system is proved by using semigroup approach. Next, by combining an orthonormal basis decomposition with frequency-multiplier techniques using appropriate cut-off functions, we prove a polynomial energy decay rate of type $t^{-1}$.\\

\section{Stability of a 2D piezoelectric beam on a rectangular domain}
\noindent In this section, we study the stability result of a $2D$ piezoelectric beam with magnetic effect on a rectangular domain.
\subsection{Well-Posedness} The energy of System $\eqref{1PIEZO-2D}-\eqref{Piezo-Initial}$, is given by 
$$
E_1(t)=\frac{1}{2}\int_{\Omega_1}(\rho \abs{v_t}^2+\alpha_1\abs{\nabla v}^2)dX+\frac{1}{2}\int_{\Omega_1}\left(\mu\abs{p_t}^2+\beta\abs{\gamma\nabla v-\nabla p}^2\right)dX.
$$
\begin{lemma}\label{denergy}
Let $U=(v,v_t,p,p_t)$ be a regular solution of system \eqref{1PIEZO-2D}-\eqref{Piezo-Initial}. Then, the energy $E_1(t)$ satisfies the following estimation 
\begin{equation}\label{denergy1}
\frac{d}{dt}E_1(t)=-\int_{\Omega_1}d(x)\abs{v_t(X,t)}^2dX.
\end{equation}	
\end{lemma}
\begin{proof}
Multiplying the first and the second equation of \eqref{1PIEZO-2D} by $\overline{u_t}$ et $\overline{y_t}$ respectively, integrating by parts over $\Omega_1$, we get 

\begin{equation}\label{denergy2}
\frac{1}{2}\frac{d}{dt}\left(\rho\int_{\Omega_1}\abs{v_t}^2 dX +\alpha\int_{\Omega_1}\abs{\nabla v}^2 dX \right)-\gamma\beta\Re\left(\int_{\Omega_1}\nabla p\cdot \nabla\overline{v_t} dX \right)+\int_{\Omega_1}d(x)\abs{v_t(X,t)}^2dX=0	
\end{equation}
and
\begin{equation}\label{denergy3}
\frac{1}{2}\frac{d}{dt}\left(\mu\int_{\Omega_1}\abs{y_t}^2dX+\beta\int_{\Omega_1}\abs{\nabla y}^2dX\right)-\gamma\beta\Re\left(\int_{\Omega_1}\nabla u\cdot \nabla\overline{y_t}dX\right)=0.	
\end{equation}
Adding \eqref{denergy2} and \eqref{denergy3}, and using the fact that $\alpha=\alpha_1+\gamma^2\beta$, we get the desired equation \eqref{denergy1}. The proof has been completed.	
\end{proof}

\noindent From \eqref{denergy1}, it follows that System \eqref{1PIEZO-2D}-\eqref{Piezo-Initial} is dissipative.  Now, let us define the energy space $\mathcal{H}$ by 
\begin{equation}\label{Hilbert}
\mathcal{H}_{\square}=\left(W_1\times L^2(\Omega_1)\right)^2,	
\end{equation}
where $W_1=\{f\in H^1(\Omega_{1});\ f=0\ \text{on}\ \Gamma_{0,1}\}$. $\mathcal{H}_{\square}$ is a Hilbert space, equipped with the inner product defined by 
\begin{equation}\label{InnerProduct}
\left<U,U_1\right>_{\mathcal{H}_{\square}}=\int_{\Omega_1}(\alpha_1\nabla v\cdot\nabla\overline{v_1}+\rho z\overline{z_1}+\beta\left(\gamma\nabla v-\nabla p\right)\cdot \left(\gamma\nabla \overline{v_1}-\nabla\overline{p_1}\right)+\mu q\overline{q_1})dX,
\end{equation}
for all $U=(v,z,p,q)^{\top}$ and $U_1=(v_1,z_1,p_1,q_1)^{\top}$ in $\mathcal{H}_{\square}$. The expression $\|\cdot\|_{\mathcal{H}_{\square}}$ will denote the corresponding norm. We define the unbounded linear operator $\mathcal{A}_{\square}:D(\mathcal{A}_\square)\subset \mathcal{H}_{\square}\rightarrow \mathcal{H}_{\square}$ by 
$$
D\left(\mathcal{A}_{\square}\right):= \left\{U:=(v,z,p,q)\in \mathcal{H}_{\square}; z,q\in W_1,\ \Delta v,\Delta z\in L^2(\Omega_1)\quad \text{and}\quad \frac{\partial v}{\partial \nu}=\frac{\partial p}{\partial \nu}=0\ \text{on}\ \Gamma_{1,1}\right\}.
$$
and 
\begin{equation}\label{Operator1}
\mathcal{A}_{\square}(v,z,p,q)=\left(z,\frac{1}{\rho}\left(\alpha \Delta v-\gamma\beta \Delta p-d\,p\right), q, \frac{1}{\mu}\left(\beta\Delta p-\gamma\beta \Delta v\right)\right). 
\end{equation}
If $U=(v,v_t,p,p_t)^{\top}$ is the state of System $\eqref{1PIEZO-2D}-\eqref{Piezo-Initial}$, then this system is transformed into the first order evolution equation on the Hilbert space $\mathcal{H}_{\square}$ given by 
\begin{equation}\label{Abstract1}
U_t=\mathcal{A}_{\square}U,\quad U(0)=U_0,	
\end{equation}
where $U_0=(v_0,v_1,p_0,p_1)^{\top}$. It is easy to see that for all $U=(v,z,p,q)\in D(\mathcal{A}_{\square})$, we have 
\begin{equation}\label{REAUU}
\Re\left(\left<\mathcal{A}_{\square}U,U\right>_{\mathcal{H}_{\square}}\right)=-\int_{\Omega_1}d(x)\abs{z(X)}^2dX\leq 0,
\end{equation}
which implies that $\mathcal{A}_{\square}$ is dissipative. Now, let $F=(f_1,f_2,f_3,f_4)\in \mathcal{H}$,  using the Lax-Milgram Theorem, one proves  the existence of $U\in D(\mathcal{A}_{\square})$, solution of the equation
$$
-\mathcal{A}_{\square}U=F.
$$
Then, the unbounded linear operator $\mathcal{A}_{\square}$ is $m-$dissipative in the energy space 
$\mathcal{H}_{\square}$ and consequently $0\in \rho(\mathcal{A}_{\square})$. Thus, $\mathcal{A}_{\square}$ generates a $C_0-$semigroup of contractions $\left(e^{t\mathcal{A}_{\square}}\right)_{t\geq 0}$ following the Lumer-Phillips theorem. The solution of the Cauchy problem \eqref{Abstract1} admits the following representation 
$$
U(t)=e^{t\mathcal{A}_{\square}}U_0,\quad t\geq 0,
$$
which leads to the well-posedness of \eqref{Abstract1}. Hence, we have the following result.
\begin{theoreme}
Let $U_{0}\in \mathcal{H}_{\square}$, then System \eqref{Abstract1} admits a unique weak solution $U$ satisfying 
$$
U\in C^0(\mathbb{R}^+,\mathcal{H}_{\square}).
$$
Moreover, if $U_0\in D(\mathcal{A}_{\square})$, then Problem \eqref{Abstract1} admits a unique strong solution $U$ satisfying 
$$
U\in C^1(\mathbb{R}^{+},\mathcal{H}_{\square})\cap C^0(\mathbb{R}^+,D(\mathcal{A}_{\square})). 
$$ 	
\end{theoreme}
\subsection{Polynomial Stability}
This subsection is devoted to showing the polynomial stability of System \eqref{1PIEZO-2D}-\eqref{Piezo-Initial}. Our main result in this subsection is the following theorem. 
\begin{theoreme}\label{Thm-Pol}
There exists a constant $C>0$ independent of $U_0$, such that the energy of System \eqref{1PIEZO-2D}-\eqref{Piezo-Initial} satisfies the following estimation 
\begin{equation}\label{ENE-EST1}
E(t)\leq \frac{C}{t}\|U_0\|^2_{D(\mathcal{A}_{\square})},\quad \forall t>0,\quad \forall U_0\in D(\mathcal{A}_{\square}). 	
\end{equation}	
\end{theoreme}
\noindent To prove this theorem, let us first introduce the following sufficient and necessary condition on the polynomial stability of a semigroup proposed by  Borichev-Tomilov in \cite{Borichev01} (see also \cite{Batty01}, \cite{RaoLiu01},
and the recent paper   \cite{ROZENDAAL2019359}).
\begin{theoreme}\label{GENN}
	{\rm
Assume that $A$ is the generator of a strongly continuous semigroup of contractions $\left(e^{t A}\right)_{t\geq 0}$ on a Hilbert space $H$. If 
\begin{equation}\label{iRrhoA}
i\R\subset \rho(A),	
\end{equation}
then for a fixed $\ell>0$ the following conditions are equivalent 
\begin{equation}\label{Condition2}
\limsup_{\la\in \R,\ |\la|\to \infty}\frac{1}{\abs{\la}^{\ell}}\|(i\la I-A)^{-1}\|_{\mathcal{L}(H)}<\infty. 
\end{equation}
\begin{equation}
\|e^{tA}X_0\|_H^2\leq \frac{C}{t^{\frac{2}{\ell}}}\|X_0\|^2_{D(A)},\ \ X_0\in D(A),\ \text{for some}\ C>0.
\end{equation}}
\end{theoreme}
\noindent According to Theorem \ref{GENN}, to prove Theorem \ref{Thm-Pol}, we need to prove that \eqref{iRrhoA} and \eqref{Condition2} hold, where $\ell=2$. For the technique, we use the orthonormal basis decomposition.  To this aim, let $e_j(y)=\sqrt{2}\sin(\xi_jy)$, where $\xi_j=\frac{(2j+1)\pi}{2}$, $j\in \mathbb{N}^{\ast}$ and $y\in (0,1)$. We may expand $v$ into a series of the form 
\begin{equation}\label{seriesv}
v(X)=\sum_{j=1}^{\infty} v_j(x)e_j(y),\quad (x,y)\in \Omega_1. 	
\end{equation}
Similarly, $z$, $p$ and $q$ can be expanded into a series of the same form as that in \eqref{seriesv} with, respectively, the coefficients $v_j(x)$, $z_j(x)$, $p_j(x)$ and $q_j(x)$. The  energy  Hilbert space \eqref{Hilbert} is given by 
\begin{equation}\label{NewH}
\mathcal{H}_{\square}=(\widehat{W}_1\times L^2(0,1))^2
\end{equation}
where 
$$
\widehat{W}_1=\left\{f\in H(0,1);\quad f(0)=0\right\}
$$
equipped with the following norm 
$$
\|U\|_{\widehat{\mathcal{H}_{\square}}}^2=\sum_{j=1}^{\infty}\left(\alpha_1(\|v_j'\|^2+\xi_j^2\|v_j\|^2)+\rho\|z_j\|^2+\beta\left(\|\gamma v_j'-p_j'\|^2+\xi_j^2\|\gamma v_j-p_j\|^2\right)+\mu \|q_j\|^2\right),
$$
where $\|\cdot\|:=\|\cdot\|_{L^2(0,1)}$ and (for later) $\|\cdot\|_{\infty}:=\|\cdot\|_{L^{\infty}(0,1)}$. This gives rise to the functions 
\begin{equation}\label{expansion-series}
(v_j,z_j,p_j,q_j)\in \left(\left(H^2(0,1)\cap \widehat{W}_1 \right)\times \widehat{W}_1 \right)^2\quad \text{and}\quad v_j'(1)=p_j'(1)=0,
\end{equation}
where " $'$ " represents the derivative with respect to $x$. The operator $\mathcal{A}_{\square}$ defined in \eqref{Operator1} can be written as 
\begin{equation}\label{Operator11}
\mathcal{A}_{\square}\begin{pmatrix}
v_j\\ z_j\\ p_j\\ q_j  	
\end{pmatrix}=\begin{pmatrix}
z_j\\ \frac{1}{\rho}\left(\alpha( v_j^{''}-\xi_j^2v_j)-\gamma \beta (p_j^{''}-\xi_j^2p_j)-d z_j\right)\\  q_j\\ \frac{1}{\mu}\left(\beta (p_j^{''}-\xi_j^2p_j)-\gamma\beta(v_j^{''}-\xi_j^2v_j) \right).
\end{pmatrix}.
\end{equation}
\begin{pro}\label{First-condition-pol}
$i\mathbb{R}\subset \rho(\mathcal{A}_{\square})$. 	
\end{pro}
\begin{proof}
To prove $i\mathbb{R}\subset \rho(\mathcal{A}_{\square})$ it is sufficient to prove that $\sigma(\mathcal{A}_{\square})\cap i\mathbb{R}=\emptyset$. Since the resolvent of $\mathcal{A}_{\square}$ is compact in $\mathcal{H}_{\square}$ then $\sigma(\mathcal{A}_{\square})=\sigma_p(\mathcal{A}_{\square})$. In the previous section,  we already proved that  $0\in \rho(\mathcal{A}_{\square})$. It remains to show that $\sigma(\mathcal{A}_{\square})\cap i\mathbb{R}^{\ast}=\emptyset$. For this aim, suppose by contradiction that there exists a real number $\lambda\neq 0$ and $U=(v,z,p,q)^{\top}\in D(\mathcal{A}_{\square})\backslash \{0\}$ such that 
\begin{equation}\label{KER1}
\mathcal{A}_{\square}U=i\la U.	
\end{equation}
Using the orthonormal basis decomposition, \eqref{Operator11} and detailing \eqref{KER1}, we get the following system 
\begin{eqnarray}
z_j&=&i\la v_j,\label{KER2}\\
i\la \rho z_j-\alpha( v_j^{''}-\xi_j^2v_j)+\gamma \beta (p_j^{''}-\xi_j^2p_j)+d z_j&=&0,\label{KER3}\\
q_j&=&i\la p_j,\label{KER4}\\
i\la \mu q_j-\beta (p_j^{''}-\xi_j^2p_j)+\gamma\beta(v_j^{''}-\xi_j^2v_j)&=&0.\label{KER5}
\end{eqnarray}
From \eqref{REAUU} and \eqref{KER1}, we have 
\begin{equation}\label{KER6}
0=\Re\left(i\la \|U\|_{\mathcal{H}_{\square}}\right)=\Re\left(\left<\mathcal{A}_{\square}U,U\right>_{\mathcal{H}_{\square}}\right)=-\int_{\Omega_1}d(x)\abs{z(X)}^2dX=-\int_0^1d(x)\abs{z_j(x)}^2dx\leq 0. 	
\end{equation}
Thus, from \eqref{KER2}, \eqref{KER6} and the fact that $\la \neq 0$, we have 
\begin{equation}\label{KER7}
d(x)z_j(x)=0\ \text{in}\ (0,1)\quad \text{and consequently}\ z_j=v_j=0\ \text{in}\ (a,b). 
\end{equation}
Using the fact that $\alpha=\alpha_1+\gamma^2\beta$ and \eqref{KER7} in \eqref{KER3}, we get 
\begin{equation}\label{KER8}
i\la\rho z_j-\alpha_1(v_j^{''}-\xi_j^2v_j)-\gamma\left(\gamma\beta( v_j^{''}-\xi_j^2v_j)-\beta (p_j^{''}-\xi_j^2p_j)\right)=0,\quad \text{in}\quad (0,1).
\end{equation}
Combining \eqref{KER8} and \eqref{KER5}, we get 
\begin{equation}\label{KER9}
i\la \left(\rho z_j+\gamma\mu q_j\right)-\alpha_1\left(v_j^{''}-\xi_j^2v_j\right)=0,\quad \text{in}\quad (0,1).	
\end{equation}
Using \eqref{KER7} in \eqref{KER9} and the fact that $\la\neq 0$,  then using \eqref{KER4}, we get 
\begin{equation}
q_j=p_j=0\quad \text{in}\quad (a,b). 	
\end{equation}
Since $v_j,p_j\in H^2(a,b)\subset C^1([a,b])$, we get 
\begin{equation}\label{KER10}
v_j(\xi)=v_j'(\xi)=p_j(\xi)=p_j'(\xi)=0\quad \text{where}\quad \xi\in \{a,b\}.	
\end{equation}
Inserting \eqref{KER2} and \eqref{KER4} in \eqref{KER9}, then combining with \eqref{KER5},  we get the following system
\begin{eqnarray}
v_j^{''}-\xi_j^2v_j&=&-\frac{\la^2}{\alpha_1}\left(\rho v_j+\gamma\mu p_j\right),\quad \text{in}\quad (0,1)\label{KER11}\\
p_j^{''}-\xi_j^2p_j&=&-\frac{\la^2}{\alpha_1}\left(\gamma\rho v_j+\mu\frac{\alpha}{\beta}p_j\right),\quad \text{in}\quad (0,1)\label{KER12}.	
\end{eqnarray}
Let $\widetilde{U}=(v_j,v_j',p_j,p_j')^{\top}$. From \eqref{KER10}, we get $\widetilde{U}(b)=0$. Now, system \eqref{KER11}-\eqref{KER12} can be written in $(b,L)$ as the following 
\begin{equation}\label{Ker13}
\widetilde{U}_x=B\widetilde{U}\quad \text{in}\quad (b,L),	
\end{equation}
where 
$$
B=\begin{pmatrix}
0&1&0&0\\
\frac{\alpha_1\xi_j^2-\rho\la^2}{\alpha_1}&0&-\frac{\gamma\mu\la^2}{\alpha_1}&0\\
0&0&0&1\\
-\frac{\la^2\rho\gamma}{\alpha_1}&0&\frac{\alpha_1\beta \xi_j^2-\la^2\mu\alpha}{\alpha_1\beta}&0
\end{pmatrix}.
$$
The solution of the differential equation \eqref{Ker13} is given by 
\begin{equation}\label{Ker14}
\widetilde{U}(x)=e^{B(x-b)}\widetilde{U}(b)=0\quad \text{in}\quad (b,L).	
\end{equation}
In the same way, we prove that $\widetilde{U}=0$ in $(0,a)$. Consequently, we get $v_j=p_j=0$ in $(0,L)$ therefore $U=0$ in $(0,L)$. The proof is thus complete.   
\end{proof}

$\newline$
\noindent As condition \eqref{iRrhoA} is already proved in Proposition \ref{First-condition-pol}, we only need to prove condition \eqref{Condition2}. Here, we use a contradiction argument. Namely, suppose that \eqref{Condition2} is false, then there exists
$$
\left\{(\la_n,U^{(n)}:=(v^{(n)},z^{(n)},p^{(n)},q^{(n)}))\right\}_{n\geq 1}\subset \mathbb{R}^{\ast}_+\times D(\mathcal{A}_{\square}),
$$
with 
\begin{equation}\label{Cond-Pol1}
\la_n\rightarrow \infty\ \text{as}\ n\to \infty\ \text{and}\ \|U^{(n)}\|_{\mathcal{H}_{\square}}=\|(v^{(n)},z^{(n)},p^{(n)},q^{(n)})\|_{\mathcal{H}_{\square}}=1,\quad \forall n\in \mathbb{N},
\end{equation}
such that 
\begin{equation}\label{Cond-Pol2}
\la_n^{\ell}(i\la_n-\mathcal{A})U^{(n)}=F^{(n)}:=(f^{1,(n)},f^{2,(n)},f^{3,(n)},f^{4,(n)})\to 0\quad \text{in}\quad \mathcal{H}_{\square},\ \text{as} \ n\to \infty. 	
\end{equation}
For simplicity, we drop the index $n$. Detailing equation \eqref{Cond-Pol2}, we get 
\begin{equation}\label{1DETAIL}
\left\{\begin{array}{rll}
i\la v-z&=&\displaystyle \la^{-2}f^1,\\[0.1in]
i\la \rho z-\alpha \Delta v+\gamma\beta\Delta p+d\, z&=&\displaystyle \la^{-2}\rho f^2,\\[0.1in]
i\la p-q&=&\displaystyle \la^{-2}f^3,\\[0.1in]
i\la \mu q-\beta \Delta p+\gamma\beta \Delta v&=&\displaystyle \la^{-2}\mu f^4.\\[0.1in]
\end{array}
\right.	
\end{equation}
Using the orthonormal basis decomposition, system \eqref{1DETAIL} turns into the system of one-dimensional equations 
\begin{eqnarray}
i\la v_j-z_j&=&\frac{f^1_j}{\la^2},\label{DETAIL1}\\
i\la \rho z_j-\alpha(v_j^{''}-\xi_j^2v_j)+\gamma\beta(p_j^{''}-\xi_j^2p_j)+d\, z_j&=&\frac{\rho f^2_j}{\la^2},\label{DETAIL2}\\
i\la p_j-q_j&=&\frac{f^3_j}{\la^2},\label{DETAIL3}\\
i\la \mu q_j-\beta (p_j^{''}-\xi_j^2p_j)+\gamma\beta (v_j^{''}-\xi_j^2v_j)&=&\frac{\mu f^4_j}{\la^2},\label{DETAIL4}.	
\end{eqnarray}
Inserting \eqref{DETAIL1} and \eqref{DETAIL3} in \eqref{DETAIL2} and \eqref{DETAIL4}, we get the following system 
\begin{eqnarray}
\left(\la^2\rho-\alpha\xi_j^2\right)v_j+\alpha v_j^{''}-\gamma\beta\left(p_j^{''}-\xi_j^2p_j\right)-i\la d v_j&=&F^1_j,\label{COMB1-DETAIL}\\
\left(\la^2\mu-\beta\xi_j^2\right)p_j+\beta p_j^{''}-\gamma\beta\left(v_j^{''}-\xi_j^2v_j\right)&=&F^2_j,\label{COMB2-DETAIL}
\end{eqnarray}
where 
\begin{equation}\label{F1F2}
F_j^1:=-\left(\frac{\rho f_j^2}{\la^2}+\frac{i\rho f_j^1}{\la}+\frac{d f^1_j}{\la^2}\right)\quad \text{and}\quad 	F_j^2:=-\left(\frac{\mu f_j^4}{\la^2}+\frac{i\mu f_j^3}{\la}\right).
\end{equation}
Using the fact that $\alpha=\alpha_1+\gamma^2\beta$ in \eqref{COMB1-DETAIL}, we get 
\begin{equation}\label{COMB3-DETAIL}
\left(\la^2\rho-\alpha_1\xi_j^2\right)v_j+\alpha_1v_j^{''}+\gamma\left(\gamma\beta(v_j^{''}-\xi_j^2v_j)-\beta(p_j^{''}-\xi_j^2p_j)\right)-i\la dv_j=F_j^1.	
\end{equation}
Now, combining \eqref{COMB2-DETAIL} and \eqref{COMB3-DETAIL}, we get 
\begin{equation}\label{COMB4-DETAIL}
\alpha_1\left(v_j^{''}-\xi_j^2v_j\right)=-\la^2\rho v_j-\gamma\la^2\mu p_j+i\la d v_j+F_j^1+\gamma F_j^2.	
\end{equation}
Inserting \eqref{COMB4-DETAIL} in \eqref{COMB2-DETAIL}, we get the following system
\begin{eqnarray}
\left(\la^2\rho-\alpha_1\xi_j^2\right)v_j+\alpha_1v_j^{''}+\gamma\mu\la^2p_j-i\la d v_j&=&F_j^3,\label{COMB5-DETAIL}\\
\left(\la^2\mu\alpha-\alpha_1\beta\xi_j^2\right)p_j+\alpha_1\beta p_j^{''}+\rho\gamma\beta\la^2v_j-i\la \gamma\beta d v_j&=&F_j^4,\label{COMB6-DETAIL}\\
v_j(0)=p_j(0)=v_j'(1)=p_j'(1)&=&0\label{COMB7-DETAIL},
\end{eqnarray}
where 
\begin{equation}\label{F3F4}
F_j^3=F_j^1+\gamma F_j^2\quad \text{and}\quad F_j^4=\alpha F_j^2+\gamma \beta F_j^1.	
\end{equation}
Before going on, let us first give the consequence of the dissipativeness property of the solution $\left(v_j,z_j,p_j,q_j\right)$ of the system \eqref{DETAIL1}-\eqref{DETAIL4}.
\begin{lemma}\label{INFO1-SQUARE}
The solution $(v,z,p,q)$ of system \eqref{1DETAIL} satisfies the following estimations 
\begin{equation}\label{EQ1-INFO1-SQUARE}
\sum_{j=1}^{\infty}\|\sqrt{d}z_j\|^2=o(\la^{-2})\quad \text{and}\quad \sum_{j=1}^{\infty}\|\la \sqrt{d}v_j\|^2=o(\la^{-2}). 	
\end{equation}
\end{lemma}
\begin{proof}
First, taking the inner product of \eqref{Cond-Pol2} with $U$ in $\mathcal{H}$, using the fact that $\|U\|_{\mathcal{H}}=1$ and $\|F\|_{\mathcal{H}}=o(1)$, we get 
\begin{equation}\label{EQ2-INFO1-SQUARE}
\|\sqrt{d}z\|^2_{L^2(\Omega_1)}=-\Re\left(\left<\mathcal{A}_{\square}U,U\right>_{\mathcal{H}_{\square}}\right)	=\Re\left(\left<(i\la I-\mathcal{A}_{\square})U,U\right>_{\mathcal{H}_{\square}}\right)=\la^{-2}\Re\left(\left<F,U\right>_{\mathcal{H}_{\square}}\right)=o(\la^{-2}). 
\end{equation}
 Thus, by the orthonormal basis decomposition, we get the first estimation in \eqref{EQ1-INFO1-SQUARE}. Now, multiplying \eqref{DETAIL1} by $\sqrt{d}$ and using the first estimation in \eqref{EQ1-INFO1-SQUARE} and that $\|F\|_{\mathcal{H}}=o(1)$, we get the second estimation in \eqref{EQ1-INFO1-SQUARE}. The proof has been completed.  	
\end{proof}

$\newline$
\noindent For all $0<\varepsilon<\frac{b-a}{4}$, we fix the following cut-off functions 
\begin{enumerate}
\item[$\bullet$] $\theta_k\in C^2([0,1])$, $k\in \left\{1,2\right\}$ such that $0\leq \theta_k(x)\leq 1$, for all $x\in [0,1]$ and 
$$
\theta_k(x)=\left\{\begin{array}{lll}
1&\text{if}&x\in [a+k\varepsilon,b-k\varepsilon],\\
0&\text{if}&x\in [0,a+(k-1)\varepsilon]\cup [b+(1-k)\varepsilon,1]. 	
\end{array}
\right.
$$
\end{enumerate}

\begin{figure}[h]
\begin{center}
\begin{tikzpicture}
\draw[->](0,0)--(12,0);
\draw[->](0,0)--(0,4);
\node[black,below] at (0,0){\scalebox{0.75}{$0$}};
\node at (0,0) [circle, scale=0.3, draw=black!80,fill=black!80] {};
\node[black,below] at (2,0){\scalebox{0.75}{$a$}};
\node at (2,0) [circle, scale=0.3, draw=black!80,fill=black!80] {};	
\node[black,below] at (3,0){\scalebox{0.75}{$a+\varepsilon$}};
\node at (3,0) [circle, scale=0.3, draw=black!80,fill=black!80] {};
\node[black,below] at (4,0){\scalebox{0.75}{$a+2\varepsilon$}};
\node at (4,0) [circle, scale=0.3, draw=black!80,fill=black!80] {};
\node[black,below] at (7,0){\scalebox{0.75}{$b-2\varepsilon$}};
\node at (7,0) [circle, scale=0.3, draw=black!80,fill=black!80] {};
\node[black,below] at (8,0){\scalebox{0.75}{$b-\varepsilon$}};
\node at (8,0) [circle, scale=0.3, draw=black!80,fill=black!80] {};
\node[black,below] at (9,0){\scalebox{0.75}{$b$}};
\node at (9,0) [circle, scale=0.3, draw=black!80,fill=black!80] {};
\node[black,below] at (10,0){\scalebox{0.75}{$1$}};
\node at (10,0) [circle, scale=0.3, draw=black!80,fill=black!80] {};
\node[black,left] at (0,2){\scalebox{0.75}{$1$}};
\node at (0,2) [circle, scale=0.3, draw=black!80,fill=black!80] {};
\node[black,left] at (0,1){\scalebox{0.75}{$d_0$}};
\node at (0,1) [circle, scale=0.3, draw=black!80,fill=black!80] {};
\draw[-,red](0,0.011)--(2,0.011);
\draw[red] (2,0.011) to[out=0.40,in=180] (3,2.011) ;
\draw[-,red](3,2.011)--(8,2.011);
\draw[red] (8,2.011) to[out=0.40,in=180] (9,0.011) ;
\draw[-,red](9,0.011)--(10,0.011);
\draw[-,blue](0,0.001)--(3,0);
\draw[-,brown](0,0.01)--(2,0);
\draw[-,blue] (3,0) to[out=0.40,in=180] (4,2) ;
\draw[-,blue](4,2)--(7,2);
\draw [blue] (7,2) to[out=0.40,in=180] (8,0) ;
\draw[brown] (2,3) to [out=0.40, in=180] (5,1);
\draw[brown] (5,1) to [out=0.40, in=180] (9,3.5);
\draw[-,blue](8,0)--(10,0);
\draw[dashed] (5.1,0.99)--(0,1);
\draw[dashed] (3,2)--(3,0);
\draw[dashed] (4,2)--(4,0);
\draw[dashed] (7,0)--(7,2);
\draw[dashed] (8,0)--(8,2);
\draw[dashed] (2,3)--(2,0);
\draw[dashed] (9,3.5)--(9,0);
\node[black,right] at (12.5,4.5){\scalebox{0.75}{$\theta_1$}};
\node[black,right] at (12.5,4){\scalebox{0.75}{$\theta_2$}};
\node[black,right] at (12.5,3.5){\scalebox{0.75}{$d$}};
\draw[-,red](12,4.5)--(12.5,4.5);
\draw[-,blue](12,4)--(12.5,4);
\draw[-,brown](12,3.5)--(12.5,3.5); 
\end{tikzpicture}
\caption{Geometric description of the functions $\theta_1$, $\theta_2$ and $d$.}\label{p2-Fig2}
\end{center}	
\end{figure}
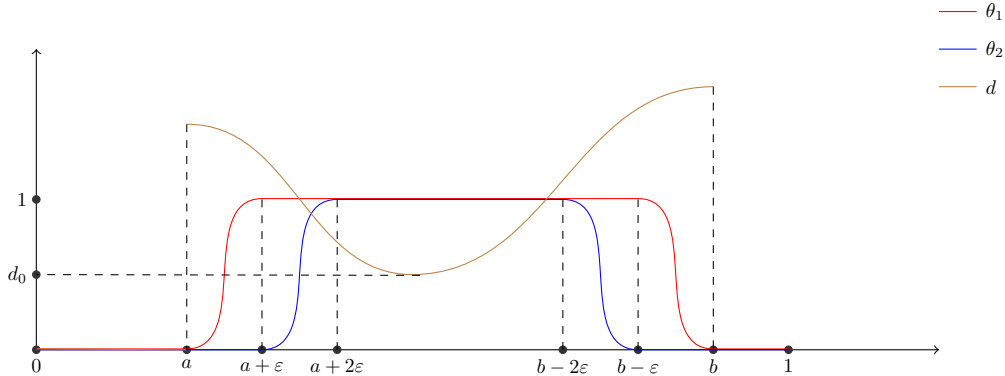

\begin{lemma}\label{INFO2-SQUARE}
The solution $(v,z,p,q)$ of system \eqref{1DETAIL} satisfies the following estimations
\begin{equation}\label{EQ1-INFO2-SQUARE}
\sum_{j=1}^{\infty}\|\la p_j\|^2_{L^2(D_{\varepsilon})}=o(\la^{-2})\quad \text{and}\quad \sum_{j=1}^{\infty}\|q_j\|^2_{L^2(D_{\varepsilon})}=o(\la^{-2}), 	
\end{equation}
where $D_{\varepsilon}:=(a+\varepsilon,b-\varepsilon)$ with a positive real number $\varepsilon$ small enough such that $\varepsilon<\frac{b-a}{4}$. 	
\end{lemma}
\begin{proof}
Multiplying \eqref{COMB5-DETAIL} by $\beta \theta_1\overline{p_j}$, using integration by parts over $(0,1)$, and the definition of $\theta_1$, we get 
\begin{equation}\label{EQ2-INFO2-SQUARE}
\begin{array}{l}
\displaystyle 
\la^2\rho\beta\int_0^1\theta_1v_j\overline{p_j}dx-\alpha_1\beta\int_0^1\theta_1\xi_j^2v_j\overline{p_j}dx-\underbrace{\alpha_1\beta\int_0^1\theta_1'v_j'\overline{p_j}dx}_{:=J_1}-\alpha_1\beta\int_0^1\theta_1v_j'\overline{p_j'}dx\\
\displaystyle 
+\gamma\mu\beta\int_0^1\theta_1\abs{\la p_j}^2dx-i\la \beta\int_0^1d\theta_1 v_j\overline{p_j}dx=\beta \int_0^1\theta_1F_j^3\overline{p_j}dx.	
\end{array}	
\end{equation}
Integrating by parts $J_1$, we get 
\begin{equation}\label{EQ3-INFO2-SQUARE}
J_1=-\alpha_1\beta\int_0^1\theta_1^{''}v_j\overline{p_j}dx-\alpha_1\beta\int_0^1\theta_1'v_j\overline{p_j}dx. 	
\end{equation}
Inserting \eqref{EQ3-INFO2-SQUARE} in \eqref{EQ2-INFO2-SQUARE}, we get 
\begin{equation}\label{EQ4-INFO2-SQUARE}
\begin{array}{l}
\displaystyle 
\la^2\rho\beta\int_0^1\theta_1v_j\overline{p_j}dx-\alpha_1\beta\int_0^1\theta_1\xi_j^2v_j\overline{p_j}dx+\alpha_1\beta\int_0^1\theta_1^{''}v_j\overline{p_j}dx+\alpha_1\beta\int_0^1\theta_1'v_j\overline{p_j'}dx\\
\displaystyle 
-\alpha_1\beta\int_0^1\theta_1v_j'\overline{p_j'}dx+\gamma\mu\beta\int_0^1\theta_1\abs{\la p_j}^2dx-i\la \beta\int_0^1d\theta_1 v_j\overline{p_j}dx=\beta \int_0^1\theta_1F_j^3\overline{p_j}dx.
\end{array}	
\end{equation}
Multiplying \eqref{COMB6-DETAIL} by $-\theta_1\overline{v_j}$ integrating by parts over $(0,1)$, we get 
\begin{equation}\label{EQ5-INFO2-SQUARE}
\begin{array}{l}
\displaystyle 	
-\la^2\mu\alpha\int_0^1\theta_1p_j\overline{v_j}dx+\alpha_1\beta\int_0^1\theta_1\xi_j^2p_j\overline{v_j}dx+\alpha_1\beta\int_0^1\theta_1'p_j'\overline{v_j}dx+\alpha_1\beta\int_0^1\theta_1p_j'\overline{v_j'}\\
\displaystyle 
-\rho\gamma\beta\int_0^1\theta_1\abs{\la v_j}^2dx+i\la \gamma\beta\int_0^1\theta_1d\abs{v_j}^2dx=-\int_0^1\theta_1F_j^4\overline{v_j}dx. 
\end{array}
\end{equation}
Adding the real part of \eqref{EQ4-INFO2-SQUARE} and that of \eqref{EQ5-INFO2-SQUARE}, then taking  the sum from $1$ to $\infty$, we get
\begin{equation}\label{EQ6-INFO2-SQUARE}
\begin{array}{l}
\displaystyle 
\la^2\left(\rho\beta-\mu\alpha\right)\sum_{j=1}^{\infty}\Re\left(\int_0^1\theta_1p_j\overline{v_j}dx\right)+\alpha_1\beta\sum_{j=1}^{\infty}\Re\left(\int_0^1\theta_1^{''}v_j\overline{p_j}dx\right)+2\alpha_1\beta\sum_{j=1}^{\infty}\Re\left(\int_0^1\theta_1'v_j\overline{p_j'}dx\right)\\
\displaystyle 
+\gamma\mu\beta\sum_{j=1}^{\infty}\int_0^1\theta_1\abs{\la p_j}^2dx-\beta\sum_{j=1}^{\infty}\Re\left(i\la \int_0^1d\theta_1 v_j\overline{p_j}dx\right)-\rho\gamma\beta\sum_{j=1}^{\infty}\int_0^1\theta_1\abs{\la v_j}^2dx\\
\displaystyle
=\beta\sum_{j=1}^{\infty}\Re\left(\int_0^1\theta_1F_j^3\overline{p_j}dx\right)-\sum_{j=1}^{\infty}\Re\left(\int_0^1\theta_1F_j^4\overline{v_j}dx\right).
\end{array}	
\end{equation}
Using Cauchy-Schwarz inequality, \eqref{EQ1-INFO1-SQUARE}, the fact that $\|U\|_{\mathcal{H}_{\square}}=1$ and $\|F\|_{\mathcal{H}_{\square}}=o(1)$, we get 
\begin{equation}\label{EQ7-INFO2-SQUARE}
\left\{\begin{array}{l}
\displaystyle 
\left|\Re\left(\int_0^1\theta_1^{''}v_j\overline{p_j}dx\right)\right|\leq \max_{x\in [0,1]}\abs{\theta_1^{''}(x)}\left(\sum_{j=1}^{\infty}\|v_j\|^2_{L^2(a,b)}\right)^{\frac{1}{2}}\left(\sum_{j=1}^{\infty}\|p_j\|^2_{L^2(a,b)}\right)^{\frac{1}{2}}=o(\la^{-3}),\\[0.1in]
\displaystyle 
\left|\Re\left(\int_0^1\theta_1'v_j\overline{p_j'}dx\right)\right|\leq \max_{x\in [0,1]}\abs{\theta_1'(x)}\left(\sum_{j=1}^{\infty}\|v_j\|^2_{L^2(a,b)}\right)^{\frac{1}{2}}\left(\sum_{j=1}^{\infty}\|p_j'\|^2_{L^2(a,b)}\right)^{\frac{1}{2}}=o(\la^{-2}),\\[0.1in]
\displaystyle 
\left|\sum_{j=1}^{\infty}\Re\left(i\la \int_0^1d\theta_1 v_j\overline{p_j}dx\right)\right|\leq \|d\|_{\infty}\left(\sum_{j=1}^{\infty}\|v_j\|^2_{L^2(a,b)}\right)^{\frac{1}{2}}\left(\sum_{j=1}^{\infty}\|\la p_j\|^2_{L^2(a,b)}\right)^{\frac{1}{2}}=o(\la^{-2}),\\[0.1in]
\displaystyle 
\left|\sum_{j=1}^{\infty}\Re\left(\int_0^1\theta_1F_j^3\overline{p_j}dx\right)\right|\leq \left(\sum_{j=1}^{\infty}\|F_j^3\|^2_{L^2(a,b)}\right)^{\frac{1}{2}}\left(\sum_{j=1}^{\infty}\|p_j\|^2_{L^2(a,b)}\right)^{\frac{1}{2}}=o(\la^{-2}),\\[0.1in]
\displaystyle 
\left|\sum_{j=1}^{\infty}\Re\left(\int_0^1\theta_1F_j^4\overline{v_j}dx\right)\right|\leq \left(\sum_{j=1}^{\infty}\|F_j^4\|^2_{L^2(a,b)}\right)^{\frac{1}{2}}\left(\sum_{j=1}^{\infty}\|v_j\|^2_{L^2(a,b)}\right)^{\frac{1}{2}}=o(\la^{-3}).
\end{array}
\right.	
\end{equation}
Inserting \eqref{EQ7-INFO2-SQUARE} in \eqref{EQ6-INFO2-SQUARE} and using \eqref{EQ1-INFO1-SQUARE}, we get 
\begin{equation}\label{EQ8-INFO2-SQUARE}
\gamma\mu\beta\sum_{j=1}^{\infty}\int_0^1\theta_1\abs{\la p_j}^2dx\leq \la^2\left|\rho\beta-\mu\alpha\right|\sum_{j=1}^{\infty}\int_0^1\theta_1\abs{p_j}\abs{v_j}dx+o(\la^{-2}).	
\end{equation}
Using Young inequality, we get 
\begin{equation}\label{EQ9-INFO2-SQUARE}
\la^2\left|\rho\beta-\mu\alpha\right|\sum_{j=1}^{\infty}\int_0^1\theta_1\abs{p_j}\abs{v_j}dx\leq \frac{\gamma\mu\beta}{2}\sum_{j=1}^{\infty}\int_0^1\theta_1\abs{\la p_j}^2dx+\frac{(\rho \beta-\mu\alpha)^2}{2\gamma\mu\beta}\sum_{j=1}^{\infty}\int_0^1\theta_1\abs{\la v_j}^2dx.	
\end{equation}
Inserting \eqref{EQ9-INFO2-SQUARE} in \eqref{EQ8-INFO2-SQUARE} and using \eqref{EQ1-INFO1-SQUARE}, we get the first estimation in \eqref{EQ1-INFO2-SQUARE}. Using the first estimation in \eqref{DETAIL3} and the fact that $\|F\|_{\mathcal{H}_{\square}}=o(1)$, we get the second estimation in \eqref{EQ1-INFO2-SQUARE}. The proof has been completed. 
\end{proof}

\begin{lemma}\label{INFO3-SQUARE}
The solution $(v,z,p,q)$ of system \eqref{1DETAIL} satisfies the following estimations
\begin{equation}\label{EQ1-INFO3-SQUARE}
\sum_{j=1}^{\infty}\left(\|v_j'\|^2_{L^2(D_{\varepsilon})}+\xi_j^2\|v_j\|^2_{L^2(D_{\varepsilon})}\right)=o(\la^{-2}). 	
\end{equation}
\end{lemma}
\begin{proof}
Multiplying \eqref{COMB5-DETAIL} by $-\theta_1\overline{v_j}$, integrating by parts over $(0,1)$, we get 
\begin{equation}\label{EQ2-INFO3-SQUARE}
\begin{array}{l}
\displaystyle 
-\rho\int_0^1\theta_1\abs{\la v_j}^2dx+\alpha_1\xi_j^2\int_0^1\theta_1\abs{v_j}^2dx+\alpha_1\int_0^1\theta_1\abs{v_j'}^2dx+\alpha_1\int_0^1\theta_1'v_j'\overline{v_j}dx\\
\displaystyle 
-\gamma\mu\la^2\int_0^1\theta_1p_j\overline{v_j}dx+i\la \int_0^1\theta_1 d\abs{v_j}^2dx=-\int_0^1\theta_1F_j^3\overline{v_j}dx. 
\end{array}	
\end{equation}
Taking the sum on $j$ in \eqref{EQ2-INFO3-SQUARE} and using \eqref{EQ1-INFO1-SQUARE}, we get 
\begin{equation}\label{EQ3-INFO3-SQUARE}
\begin{array}{l}
\displaystyle
\alpha_1\sum_{j=1}^{\infty}\int_0^1\theta_1(\abs{v_j'}^2+\xi_j^2\abs{v_j}^2)dx=- \alpha_1\sum_{j=1}^{\infty}\int_0^1\theta_1'v_j'\overline{v_j}dx+\gamma\mu\sum_{j=1}^{\infty}\la^2\int_0^1\theta_1p_j\overline{v_j}dx\\[0.1in]
\displaystyle 
-\sum_{j=1}^{\infty}\int_0^1\theta_1F_j^3\overline{v_j}dx+o(\la^{-2}).
\end{array}	
\end{equation}
Using Cauchy-Schwarz inequality, the definition of the function $\theta_1$, the fact that $\|U\|_{\mathcal{H}_{\square}}=1$, $\|F\|_{\mathcal{H}_{\square}}=o(1)$, \eqref{EQ1-INFO1-SQUARE} and \eqref{EQ1-INFO2-SQUARE}, we get 
\begin{equation*}
\left|\sum_{j=1}^{\infty}\int_0^1\theta_1'v_j'\overline{v_j}dx\right|=o(\la^{-2}),\quad \left|\sum_{j=1}^{\infty}\la^2\int_0^1\theta_1p_j\overline{v_j}dx\right|=o(\la^{-2})\quad \text{and}\quad  \left|\sum_{j=1}^{\infty}\int_0^1\theta_1F_j^3\overline{v_j}dx\right|=o(\la^{-3}). 
\end{equation*}
Inserting the above estimations in \eqref{EQ3-INFO3-SQUARE}, we get \eqref{EQ1-INFO3-SQUARE}. The proof has been completed. 	
\end{proof}
\begin{lemma}\label{INFO4-SQUARE}
The solution $(v,z,p,q)$ of system \eqref{1DETAIL} satisfies the following estimations
\begin{equation}\label{EQ1-INFO4-SQUARE}
\sum_{j=1}^{\infty}\left(\|p_j'\|^2_{L^2(D_{2\varepsilon})}+\xi_j^2\|p_j\|^2_{L^2(D_{2\varepsilon})}\right)=o(\la^{-2}),
\end{equation}
where $D_{2\varepsilon}:=(a+2\varepsilon,b-2\varepsilon)$ with a positive real number $\varepsilon$ small enough such that $\varepsilon<\frac{b-a}{4}$. 	
\end{lemma}
\begin{proof}
Multiplying \eqref{COMB6-DETAIL} by $-\theta_2\overline{p_j}$, integrating by parts over $(0,1)$, we get 
\begin{equation*}
\begin{array}{l}
\displaystyle 
-\mu\alpha\int_0^1\theta_2\abs{\la p_j}^2dx+\alpha_1\beta\xi_j^2\int_0^1\theta_2\abs{p_j}^2dx+\alpha_1\beta\int_0^1\theta_2\abs{p_j'}^2dx+\alpha_1\beta\int_0^1\theta_2'p_j'\overline{p_j}dx\\
\displaystyle 
-\rho\gamma\beta\la^2\int_0^1\theta_2v_j\overline{p_j}dx+i\la\gamma\beta\int_0^1d\theta_2v_j\overline{p_j}dx=-\int_0^1\theta_2F_j^4\overline{p_j}dx. 
\end{array}
\end{equation*}
Taking the sum on $j$ in the above equation and using \eqref{EQ1-INFO2-SQUARE}, we get 
\begin{equation}\label{EQ2-INFO4-SQUARE}
\begin{array}{l}
\displaystyle 
\alpha_1\beta\sum_{j=1}^{\infty}\int_0^1\theta_2(\abs{p_j'}+\xi_j^2\abs{p_j}^2)dx= -\alpha_1\beta\sum_{j=1}^{\infty}\int_0^1\theta_2'p_j'\overline{p_j}dx+\rho \gamma\beta \la^2\sum_{j=1}^{\infty}\int_0^1\theta_2v_j\overline{p_j}dx\\[0.1in]
\displaystyle 
-i\la \gamma\beta\sum_{j=1}^{\infty}\int_0^1d\theta_2v_j\overline{p_j}dx-\sum_{j=1}^{\infty}\int_0^1\theta_2F_j^4\overline{p_j}dx+o(\la^{-2}).
\end{array}
\end{equation}
Using Cauchy-Schwarz inequality, the definition of the function $\theta_2$, the fact that $\|U\|_{\mathcal{H}_{\square}}=1$, $\|F\|_{\mathcal{H}_{\square}}=o(1)$, \eqref{EQ1-INFO1-SQUARE} and \eqref{EQ1-INFO2-SQUARE}, we get 
\begin{equation*}
\left\{\begin{array}{ll}
\displaystyle 
\left|\sum_{j=1}^{\infty}\int_0^1\theta_2'p_j'\overline{p_j}dx\right|=o(\la^{-2}),&\displaystyle 
\left|\la^2\sum_{j=1}^{\infty}\int_0^1\theta_2v_j\overline{p_j}dx\right|=o(\la^{-2}),\\[0.1in]
\displaystyle 
\left|\la\sum_{j=1}^{\infty}\int_0^1d\theta_2v_j\overline{p_j}dx\right|=o(\la^{-3}),&\displaystyle 
 \left|\sum_{j=1}^{\infty}\int_0^1\theta_2F_j^4\overline{p_j}dx\right|=o(\la^{-3}).
\end{array}
\right.
\end{equation*}
Inserting the above estimations in \eqref{EQ2-INFO4-SQUARE} and using the definition of the function $\theta_2$, we get \eqref{EQ1-INFO4-SQUARE}. The proof has been completed.
\end{proof} 
\begin{lemma}\label{INFO5-SQUARE}
Let $h\in C^{\infty}([0,1])$ such that $h(0)=h(1)=0$. The solution $(v,z,p,q)$ of system \eqref{1DETAIL} satisfies the following estimation
\begin{equation}\label{EQ1-INFO5-SQUARE}
\begin{array}{l}
\displaystyle 
\sum_{j=1}^{\infty}\int_0^1	h'(\rho \la^2-\alpha_1\xi_j^2)\abs{v_j}^2dx+\alpha_1\sum_{j=1}^{\infty}\int_0^1h'\abs{v_j'}^2dx+2\sum_{j=1}^{\infty}\Re\left(i\la \int_0^Lhdv_j\overline{v_j'}dx\right)\\
\displaystyle 
+\mu\sum_{j=1}^{\infty}\int_0^1h'\abs{\la p_j}^2dx+\beta \sum_{j=1}^{\infty}\int_0^1h'\abs{\gamma v_j'-p_j'}^2dx-\beta \sum_{j=1}^{\infty}\int_0^1h'\xi_j^2\abs{\gamma v_j-p_j}^2dx=o(\la^{-2}).
\end{array}	
\end{equation}
\end{lemma}
\begin{proof}
First, multiplying \eqref{COMB1-DETAIL} and \eqref{COMB2-DETAIL} by $-2h\overline{v_j'}$ and $-2h\overline{p_j'}$ , taking the real part, we get 
\begin{equation*}\left\{
\begin{array}{l}
\displaystyle 
-2\la^2\rho\Re\left(\int_0^1hv_j\overline{v_j'}dx\right)+2\alpha\xi_j^2\Re\left(\int_0^1hv_j\overline{v_j'}\right)-2\alpha\Re\left(\int_0^1hv_j{''}\overline{v_j'}dx\right)\\[0.1in]
\displaystyle 
+2\gamma\beta\Re\left(\int_0^1h\left(p_j^{''}-\xi_j^2p_j\right)\overline{v_j'}dx\right)+2\Re\left(i\la \int_0^1hdv_j\overline{v_j'}dx\right)=-2\Re\left(\int_0^1hF_j^1\overline{v_j'}dx\right)\\ 
\text{and}\\ 
\displaystyle 
-2\la^2\mu\Re\left(\int_0^1hp_j\overline{p_j'}dx\right)+2\beta\xi_j^2\Re\left(\int_0^1hp_j\overline{p_j'}dx\right)-2\beta\Re\left(\int_0^1hp_j{''}\overline{p_j'}dx\right)\\[0.1in]
\displaystyle
+2\gamma\beta\Re\left(\int_0^1h\left(v_j^{''}-\xi_j^2v_j\right)\overline{p_j'}dx\right)=-2\Re\left(\int_0^1hF_j^2\overline{p_j'}dx\right). 
\displaystyle 	
\end{array}
\right.
\end{equation*}
Using integration by parts in the above equations and taking the sum on $j$ from $1$ to $\infty$, we get 
\begin{equation*}
\begin{array}{l}
\displaystyle 
\sum_{j=1}^{\infty}\int_0^1h'\left(\la^2\rho-\alpha\xi_j^2\right)\abs{v_j}^2dx+\alpha\sum_{j=1}^{\infty}\int_0^1h'\abs{v_j'}^2dx+2\sum_{j=1}^{\infty}\Re\left(i\la \int_0^1hdv_j\overline{v_j'}dx\right)\\
\displaystyle 
+\underbrace{2\gamma\beta \sum_{j=1}^{\infty}\Re\left(\int_0^1h\left(p_j^{''}-\xi_j^2p_j\right)\overline{v_j'}dx\right)}_{:=J_2}=-2\sum_{j=1}^{\infty}\Re\left(\int_0^1hF_j^1\overline{v_j'}dx\right)	
\end{array}
\end{equation*}
and
\begin{equation*}
\sum_{j=1}^{\infty}\int_0^1h'(\la^2\mu-\beta\xi_j^2)\abs{p_j}^2dx+\beta\sum_{j=1}^{\infty}\int_0^1h'\abs{p_j'}^2dx+2\gamma\beta\sum_{j=1}^{\infty}\Re\left(\int_0^1h\left(v_j^{''}-\xi_j^2v_j\right)\overline{p_j'}dx\right)=-2\sum_{j=1}^{\infty}\Re\left(\int_0^1hF_j^2\overline{p_j'}\right).
\end{equation*}
Adding the above two equations, we get 
\begin{equation}\label{EQ2-INFO5-SQUARE}
\begin{array}{l}
\displaystyle 
\sum_{j=1}^{\infty}\int_0^1h'\left(\la^2\rho-\alpha\xi_j^2\right)\abs{v_j}^2dx+\alpha\sum_{j=1}^{\infty}\int_0^1h'\abs{v_j'}^2dx+2\sum_{j=1}^{\infty}\Re\left(i\la \int_0^1hdv_j\overline{v_j'}dx\right)+J_2\\[0.1in]
\displaystyle
+\sum_{j=1}^{\infty}\int_0^1h'(\la^2\mu-\beta\xi_j^2)\abs{p_j}^2dx+\beta\sum_{j=1}^{\infty}\int_0^1h'\abs{p_j'}^2dx+2\gamma\beta\sum_{j=1}^{\infty}\Re\left(\int_0^1h\left(v_j^{''}-\xi_j^2v_j\right)\overline{p_j'}dx\right)\\[0.1in]
\displaystyle 
=-2\sum_{j=1}^{\infty}\Re\left(\int_0^1hF_j^1\overline{v_j'}dx\right)-2\sum_{j=1}^{\infty}\Re\left(\int_0^1hF_j^2\overline{p_j'}\right).	
\end{array}	
\end{equation}
Using integration by parts in $J_2$, we get 
\begin{equation}\label{EQ3-INFO5-SQUARE}
\begin{array}{l}
\displaystyle	
J_2=-2\gamma\beta\sum_{j=1}^{\infty}\Re\left(\int_0^1h'p_j'\overline{v_j'}dx\right)-2\gamma\beta\sum_{j=1}^{\infty}\Re\left(\int_0^1hp_j'\overline{v_j^{''}}dx\right)\\[0.1in]
\displaystyle 
+2\gamma\beta\sum_{j=1}^{\infty}\Re\left(\int_0^1\xi_j^2h'p_j\overline{v_j}dx\right)+2\gamma\beta\sum_{j=1}^{\infty}\Re\left(\int_0^1\xi_j^2hp_j'\overline{v_j}dx\right).
\end{array}
\end{equation}
Inserting \eqref{EQ3-INFO5-SQUARE} in \eqref{EQ2-INFO5-SQUARE} and using the fact that $\alpha=\alpha_1+\gamma^2\beta$, we get 
\begin{equation}\label{EQ4-INFO5-SQUARE}
\begin{array}{l}
\displaystyle 
\sum_{j=1}^{\infty}\int_0^1	h'(\rho \la^2-\alpha_1\xi_j^2)\abs{v_j}^2dx+\alpha_1\sum_{j=1}^{\infty}\int_0^1h'\abs{v_j'}^2dx+2\sum_{j=1}^{\infty}\Re\left(i\la \int_0^Lhdv_j\overline{v_j'}dx\right)\\
\displaystyle 
+\mu\sum_{j=1}^{\infty}\int_0^1h'\abs{\la p_j}^2dx+\underbrace{\gamma^2\beta \sum_{j=1}^{\infty}\int_0^1h'\abs{v_j'}^2dx+\beta\sum_{j=1}^{\infty}\int_0^1h'\abs{p_j'}^2dx-2\gamma\beta\sum_{j=1}^{\infty}\Re\left(\int_0^1h'p_j'\overline{v_j'}dx\right)}_{:=J_3}\\[0.1in]
\displaystyle 
-\underbrace{\left(\gamma^2\beta\sum_{j=1}^{\infty}\int_0^1h'\xi_j^2\abs{v_j}^2dx+\beta\sum_{j=1}^{\infty}\int_0^1h'\xi_j^2\abs{p_j}^2dx-2\gamma\beta\sum_{j=1}^{\infty}\Re\left(\int_0^1\xi_j^2h'p_j\overline{v_j}dx\right)\right)}_{:=J_4}\\[0.1in]
\displaystyle 
=-2\sum_{j=1}^{\infty}\Re\left(\int_0^1hF_j^1\overline{v_j'}dx\right)-2\sum_{j=1}^{\infty}\Re\left(\int_0^1hF_j^2\overline{p_j'}dx\right).	
\end{array}	
\end{equation}
It is clear that 
\begin{equation}\label{EQ5-INFO5-SQUARE}
J_3=\beta \sum_{j=1}^{\infty}\int_0^1h'\abs{\gamma v_j'-p_j'}^2dx \quad \text{and}\quad J_4=\beta \sum_{j=1}^{\infty}\int_0^1h'\xi_j^2\abs{\gamma v_j-p_j}^2dx.
\end{equation}
Inserting \eqref{EQ5-INFO5-SQUARE} in \eqref{EQ4-INFO5-SQUARE}, we get 
\begin{equation}\label{EQ6-INFO5-SQUARE}
\begin{array}{l}
\displaystyle 
\sum_{j=1}^{\infty}\int_0^1	h'(\rho \la^2-\alpha_1\xi_j^2)\abs{v_j}^2dx+\alpha_1\sum_{j=1}^{\infty}\int_0^1h'\abs{v_j'}^2dx+2\sum_{j=1}^{\infty}\Re\left(i\la \int_0^Lhdv_j\overline{v_j'}dx\right)\\
\displaystyle 
+\mu\sum_{j=1}^{\infty}\int_0^1h'\abs{\la p_j}^2dx+\beta \sum_{j=1}^{\infty}\int_0^1h'\abs{\gamma v_j'-p_j'}^2dx-\beta \sum_{j=1}^{\infty}\int_0^1h'\xi_j^2\abs{\gamma v_j-p_j}^2dx\\[0.1in]
\displaystyle 
=-2\sum_{j=1}^{\infty}\Re\left(\int_0^1hF_j^1\overline{v_j'}dx\right)-2\sum_{j=1}^{\infty}\Re\left(\int_0^1hF_j^2\overline{p_j'}dx\right).
\end{array}	
\end{equation}
Using the definition of $F^1$ and $F^2$ in \eqref{F1F2}, and the facts that $\displaystyle{\sum_{j=1}^{\infty}\|v_j'\|^2=O(1)}$, $\displaystyle{\sum_{j=1}^{\infty}\|p_j'\|=O(1)}$ and the fact $\|F\|_{\mathcal{H}_{\square}}=o(1)$, we get 
\begin{equation}\label{EQ7-INFO5-SQUARE}
\left\{\begin{array}{l}
\displaystyle 
\sum_{j=1}^{\infty}\Re\left(\int_0^1hF_j^1\overline{v_j'}dx\right)=\la^{-1}\sum_{j=1}^{\infty}\Re\left(i\rho \int_0^1hf_j^1\overline{v_j'}dx\right)+o(\la^{-2})\\
\displaystyle 
\sum_{j=1}^{\infty}\Re\left(\int_0^1hF_j^2\overline{p_j'}dx\right)=\la^{-1}\sum_{j=1}^{\infty}\Re\left(i\mu \int_0^1hf_j^3\overline{p_j'}dx\right)+o(\la^{-2}).
\end{array}
\right.
\end{equation}
Integrating by parts over $(0,1)$ in the above estimations and using the fact that $\displaystyle{\sum_{j=1}^{\infty}\|\la v_j\|^2=O(1), \sum_{j=1}^{\infty}\|\la p_j\|^2=O(1)}$ and  $\|F\|_{\mathcal{H}_{\square}}=o(1)$ in \eqref{EQ7-INFO5-SQUARE}, we get the following  estimations 
\begin{equation}\label{EQ8-INFO5-SQUARE}
\left\{\begin{array}{l}
\displaystyle 
\sum_{j=1}^{\infty}\Re\left(\int_0^1hF_j^1\overline{v_j'}dx\right)=-\underbrace{\la^{-1}\sum_{j=1}^{\infty}\Re\left(i\rho \int_0^1(h'f_j^1+h(f_j^1)')\overline{v_j}dx\right)}_{=o(\la^{-2})}+o(\la^{-2})=o(\la^{-2})\\[0.1in]
\displaystyle 
\sum_{j=1}^{\infty}\Re\left(\int_0^1hF_j^2\overline{p_j'}dx\right)=-\underbrace{\la^{-1}\sum_{j=1}^{\infty}\Re\left(i\mu \int_0^1(h'f_j^3+h(f_j^3)')\overline{p_j}dx\right)}_{=o(\la^{-2})}+o(\la^{-2})=o(\la^{-2}).
\end{array}
\right.	
\end{equation}
Inserting \eqref{EQ8-INFO5-SQUARE} in \eqref{EQ6-INFO5-SQUARE}, we get the desired estimation \eqref{EQ1-INFO5-SQUARE}. The proof has been completed. 
\end{proof}
\begin{lemma}\label{INFO6-SQUARE}
Let $h\in C^{\infty}([0,1])$ such that $h(0)=h(1)=0$. The solution $(v,z,p,q)$ of system \eqref{1DETAIL} satisfies the following estimation
\begin{equation}\label{EQ1-INFO6-SQUARE}
\begin{array}{l}
\displaystyle 
\sum_{j=1}^{\infty}\int_0^1h'(-\la^2\rho+\alpha_1\xi_j^2)\abs{v_j}^2dx+\alpha_1\sum_{j=1}^{\infty}\int_0^1h'\abs{v_j'}^2dx-\mu\sum_{j=1}^{\infty}\int_0^1h'\abs{\la p_j}^2dx\\[0.1in]
\displaystyle 
+\beta \sum_{j=1}^{\infty}\int_0^1h'\abs{\gamma v_j'-p_j'}^2dx+\beta \sum_{j=1}^{\infty}\int_0^1h'\xi_j^2\abs{\gamma v_j-p_j}^2dx+I=o(\la^{-2}), 
\end{array}
\end{equation}
where 
\begin{equation}\label{I}
\begin{array}{l}
\displaystyle 
I= \alpha\Re\left(\sum_{j=1}^{\infty}\int_0^1h^{''}v_j'\overline{v_j}dx\right)+\beta\Re\left(\sum_{j=1}^{\infty}\int_0^1h^{''}p_j'\overline{p_j}dx\right)\\
\displaystyle 
-\gamma\beta\Re\left(\sum_{j=1}^{\infty}\int_0^1h^{''}p_j'\overline{v_j}dx\right)-\gamma\beta \Re\left(\sum_{j=1}^{\infty}\int_0^1h^{''}v_j'\overline{p_j}dx\right).	
\end{array}
\end{equation}
\end{lemma}
\begin{proof}
Multiplying \eqref{COMB1-DETAIL} by $-h'\overline{v_j}$ integrating by parts over $(0,1)$ and taking the real part, we get 
\begin{equation}\label{EQ2-INFO6-Square}
\begin{array}{l}
\displaystyle 
-\rho\int_0^1h'\abs{\la v_j}^2dx+\alpha\xi_j^2\int_0^1h'\abs{v_j}^2dx+\alpha\Re\left(\int_0^1h^{''}v_j'\overline{v_j}dx\right)+\alpha\int_0^1h'\abs{v_j'}^2dx\\
\displaystyle 
+\gamma\beta\Re\left(\int_0^1(p_j^{''}-\xi_j^2p_j)h'\overline{v_j}dx\right)=-\Re\left(\int_0^1h'F_j^1\overline{v_j}dx\right).
\end{array}	
\end{equation}
Taking the sum on $j=1$ to $\infty$ in the above equation, we get 
\begin{equation}\label{EQ3-INFO6-Square}
\begin{array}{l}
\displaystyle 
-\rho\sum_{j=1}^{\infty}\int_0^1h'\abs{\la v_j}^2dx+\alpha\sum_{j=1}^{\infty}\int_0^1h'\xi_j^2\abs{v_j}^2dx+\alpha\sum_{j=1}^{\infty}\Re\left(\int_0^1h^{''}v_j'\overline{v_j}dx\right)+\alpha\sum_{j=1}^{\infty}\int_0^1h'\abs{v_j'}^2dx\\
\displaystyle 
+\underbrace{\gamma\beta\sum_{j=1}^{\infty}\Re\left(\int_0^1(p_j^{''}-\xi_j^2p_j)h'\overline{v_j}dx\right)}_{:=J_5}=-\sum_{j=1}^{\infty}\Re\left(\int_0^1h'F_j^1\overline{v_j}dx\right).
\end{array}		
\end{equation}
Multiplying \eqref{COMB2-DETAIL} by $-h'\overline{p_j}$ integrating by parts over,  $(0,1)$ taking the real part and summing the result on $j=1$ to $\infty$, we get 
\begin{equation}\label{EQ4-INFO6-Square}
\begin{array}{l}
\displaystyle 
-\mu\sum_{j=1}^{\infty}\int_0^1h'\abs{\la p_j}^2dx+\beta\sum_{j=1}^{\infty}\int_0^1h'\xi_j^2\abs{p_j}^2dx+\beta\sum_{j=1}^{\infty}\Re\left(\int_0^1h^{''}p_j'\overline{p_j}dx\right)+\beta\sum_{j=1}^{\infty}\int_0^1h'\abs{p_j'}^2dx\\
\displaystyle 
+\underbrace{\gamma\beta\sum_{j=1}^{\infty}\Re\left(\int_0^1(v_j^{''}-\xi_j^2v_j)h'\overline{p_j}dx\right)}_{:=J_6}=-\sum_{j=1}^{\infty}\Re\left(\int_0^1h'F_j^2\overline{p_j}dx\right).
\end{array}
\end{equation}
Summing \eqref{EQ3-INFO6-Square} and \eqref{EQ4-INFO6-Square}, we get 
\begin{equation}\label{EQ5-INFO6-Square}
\begin{array}{l}
\displaystyle 
-\rho\sum_{j=1}^{\infty}\int_0^1h'\abs{\la v_j}^2dx+\alpha\sum_{j=1}^{\infty}\int_0^1h'\xi_j^2\abs{v_j}^2dx+\alpha\sum_{j=1}^{\infty}\int_0^1h'\abs{v_j'}^2dx-\mu\sum_{j=1}^{\infty}\int_0^1h'\abs{\la p_j}^2dx\\
\displaystyle
+\beta\sum_{j=1}^{\infty}\int_0^1h'\xi_j^2\abs{p_j}^2dx+\beta\sum_{j=1}^{\infty}\int_0^1h'\abs{p_j'}^2dx+J_5+J_6+\alpha\sum_{j=1}^{\infty}\Re\left(\int_0^1h^{''}v_j'\overline{v_j}dx\right)\\
\displaystyle
+\beta\sum_{j=1}^{\infty}\Re\left(\int_0^1h^{''}p_j'\overline{p_j}dx\right)=-\sum_{j=1}^{\infty}\Re\left(\int_0^1h'F_j^1\overline{v_j}dx\right)-\sum_{j=1}^{\infty}\Re\left(\int_0^1h'F_j^2\overline{p_j}dx\right).
\end{array}	
\end{equation}
Integrating by parts $J_5$ and $J_6$, we get 
\begin{equation*}
\begin{array}{l}
\displaystyle 
J_5= -\gamma\beta\sum_{j=1}^{\infty}\Re\left(\int_0^1h'p_j'\overline{v_j'}dx\right)-\gamma\beta\sum_{j=1}^{\infty}\Re\left(\int_0^1h''p_j'\overline{v_j}dx\right)-\gamma\beta\sum_{j=1}^{\infty}\Re\left(\int_0^1\xi_j^2h'p_j\overline{v_j}dx\right),\\
\displaystyle 
J_6=-\gamma\beta\sum_{j=1}^{\infty}\Re\left(\int_0^1h'v_j'\overline{p_j'}dx\right)-\gamma\beta\sum_{j=1}^{\infty}\Re\left(\int_0^1h''v_j'\overline{p_j}dx\right)-\gamma\beta\sum_{j=1}^{\infty}\Re\left(\int_0^1\xi_j^2h'v_j\overline{p_j}dx\right).	
\end{array}	
\end{equation*}
Inserting the above two equations in \eqref{EQ5-INFO6-Square}, we get
\begin{equation}\label{EQ6-INFO6-Square}
\begin{array}{l}
\displaystyle 
-\rho\sum_{j=1}^{\infty}\int_0^1h'\abs{\la v_j}^2dx+\alpha\sum_{j=1}^{\infty}\int_0^1h'\xi_j^2\abs{v_j}^2dx+\alpha\sum_{j=1}^{\infty}\int_0^1h'\abs{v_j'}^2dx-\mu\sum_{j=1}^{\infty}\int_0^1h'\abs{\la p_j}^2dx\\
\displaystyle 
+\beta\sum_{j=1}^{\infty}\int_0^1h'\xi_j^2\abs{p_j}^2dx+\beta\sum_{j=1}^{\infty}\int_0^1h'\abs{p_j'}^2dx-2\gamma\beta\sum_{j=1}^{\infty}\Re\left(\int_0^1h'p_j'\overline{v_j'}dx\right)+I\\
\displaystyle
-2\gamma\beta\sum_{j=1}^{\infty}\Re\left(\int_0^1\xi_j^2h'p_j\overline{v_j}dx\right)=-\sum_{j=1}^{\infty}\Re\left(\int_0^1h'F_j^1\overline{v_j}dx\right)-\sum_{j=1}^{\infty}\Re\left(\int_0^1h'F_j^2\overline{p_j}dx\right),
\end{array}	
\end{equation}
where $I$ is defined in \eqref{I}. Using the fact that $\alpha=\alpha_1+\gamma^2\beta$ in \eqref{EQ6-INFO6-Square}, we get 
\begin{equation}\label{EQ7-INFO6-Square}
\begin{array}{l}
\displaystyle 
\sum_{j=1}^{\infty}\int_0^1h'(-\la^2\rho+\alpha_1\xi_j^2)\abs{v_j}^2dx+\alpha_1\sum_{j=1}^{\infty}\int_0^1h'\abs{v_j'}^2dx-\mu\sum_{j=1}^{\infty}\int_0^1h'\abs{\la p_j}^2dx\\
\displaystyle 
\underbrace{+\gamma^2\beta\sum_{j=1}^{\infty}\int_0^1h'\abs{v_j'}^2dx+\beta\sum_{j=1}^{\infty}\int_0^1h'\abs{p_j'}^2dx-2\gamma\beta\sum_{j=1}^{\infty}\Re\left(\int_0^1h'p_j'\overline{v_j'}dx\right)}_{:=J_7}\\
\displaystyle 
\underbrace{+\gamma^2\beta\sum_{j=1}^{\infty}\int_0^1h'\xi_j^2\abs{v_j}^2dx+\beta\sum_{j=1}^{\infty}\int_0^1h'\xi_j^2\abs{p_j}^2dx-2\gamma\beta\sum_{j=1}^{\infty}\Re\left(\int_0^1\xi_j^2h'p_j\overline{v_j}dx\right)}_{:=J_8}\\
\displaystyle 
+I=-\sum_{j=1}^{\infty}\Re\left(\int_0^1h'F_j^1\overline{v_j}dx\right)-\sum_{j=1}^{\infty}\Re\left(\int_0^1h'F_j^2\overline{p_j}dx\right).
\end{array}	
\end{equation}
It is easy to check that 
\begin{equation}\label{EQ8-INFO6-Square}
J_7=\beta \sum_{j=1}^{\infty}\int_0^1h'\abs{\gamma v_j'-p_j'}^2dx\quad \text{and}\quad J_8=\beta \sum_{j=1}^{\infty}\int_0^1h'\xi_j^2\abs{\gamma v_j-p_j}^2dx. 	
\end{equation}
Using Cauchy-Schwarz inequality and the facts $\displaystyle{\sum_{j=1}^{\infty}\|\la v_j\|^2=O(1)}$, $\displaystyle{\sum_{j=1}^{\infty}\|\la p_j\|^2=O(1)}$, the definition of $F^1$ and $F^2$ given by \eqref{F1F2}, and $\|F\|_{\mathcal{H}_{\square}}=o(1)$, we get 
\begin{equation}\label{EQ9-INFO6-Square}
\left|\sum_{j=1}^{\infty}\Re\left(\int_0^1h'F_j^1\overline{v_j}dx\right)\right|=o(\la^{-2})\quad \text{and}\quad \left|\sum_{j=1}^{\infty}\Re\left(\int_0^1h'F_j^2\overline{p_j}dx\right)\right|=o(\la^{-2}). 
\end{equation}
Inserting \eqref{EQ8-INFO6-Square} and \eqref{EQ9-INFO6-Square} in \eqref{EQ7-INFO6-Square}, we get the desired result \eqref{EQ1-INFO6-SQUARE}. The proof has been completed. 
\end{proof}
 
$\newline$
\noindent For all $0<\varepsilon<\frac{b-a}{4}$, we fix the following cut-off functions 
\begin{enumerate}
\item[$\bullet$] $h_1,h_2\in C^2([0,1])$ such that $0\leq h_1(x)\leq 1$, $0\leq h_2(x)\leq 1$, for all $x\in [0,1]$ and 
\end{enumerate}
\begin{equation*}
h_1=\left\{\begin{array}{lll}
1&\text{if}&x\in [0,a+2\varepsilon]\\
0&\text{if}&x\in [b-2\varepsilon,1],	
\end{array}
\right.
\quad \text{and}\quad  
h_2=\left\{\begin{array}{lll}
0&\text{if}&x\in [0,a+2\varepsilon]\\
1&\text{if}&x\in [b-2\varepsilon,1].	
\end{array}	
\right.
\end{equation*}

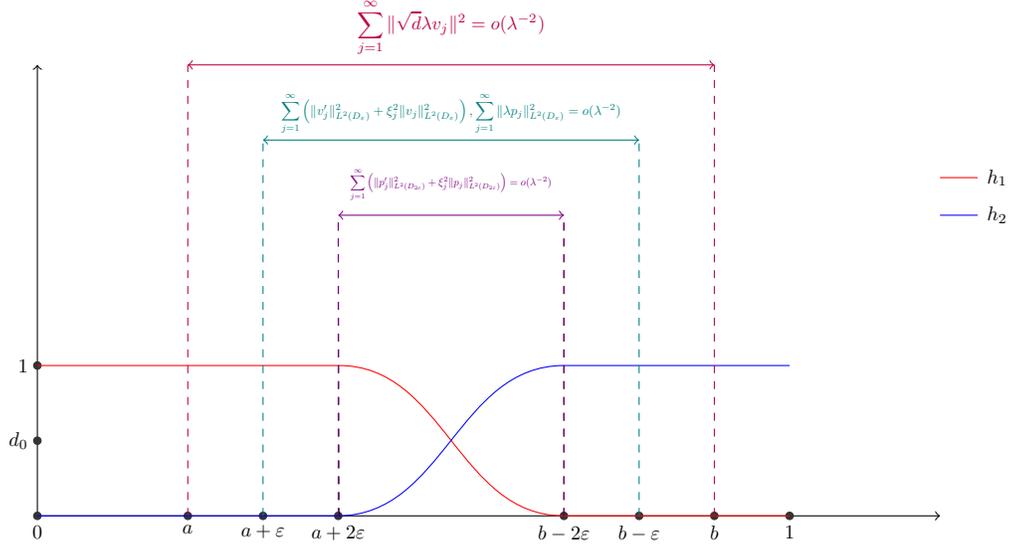
\begin{figure}[h]
\begin{center}
\begin{tikzpicture}
\draw[->](0,0)--(12,0);
\draw[->](0,0)--(0,6);
\node[black,below] at (0,0){\scalebox{0.75}{$0$}};
\node at (0,0) [circle, scale=0.3, draw=black!80,fill=black!80] {};
\node[black,below] at (2,0){\scalebox{0.75}{$a$}};
\node at (2,0) [circle, scale=0.3, draw=black!80,fill=black!80] {};	
\node[black,below] at (3,0){\scalebox{0.75}{$a+\varepsilon$}};
\node at (3,0) [circle, scale=0.3, draw=black!80,fill=black!80] {};
\node[black,below] at (4,0){\scalebox{0.75}{$a+2\varepsilon$}};
\node at (4,0) [circle, scale=0.3, draw=black!80,fill=black!80] {};
\node[black,below] at (7,0){\scalebox{0.75}{$b-2\varepsilon$}};
\node at (7,0) [circle, scale=0.3, draw=black!80,fill=black!80] {};
\node[black,below] at (8,0){\scalebox{0.75}{$b-\varepsilon$}};
\node at (8,0) [circle, scale=0.3, draw=black!80,fill=black!80] {};
\node[black,below] at (9,0){\scalebox{0.75}{$b$}};
\node at (9,0) [circle, scale=0.3, draw=black!80,fill=black!80] {};
\node[black,below] at (10,0){\scalebox{0.75}{$1$}};
\node at (10,0) [circle, scale=0.3, draw=black!80,fill=black!80] {};
\node[black,left] at (0,2){\scalebox{0.75}{$1$}};
\node at (0,2) [circle, scale=0.3, draw=black!80,fill=black!80] {};
\node[black,left] at (0,1){\scalebox{0.75}{$d_0$}};
\node at (0,1) [circle, scale=0.3, draw=black!80,fill=black!80] {};	
\draw[-,red] (0,2)--(4,2);
\draw[red] (4,2) to[out=0.40,in=180] (7,0);
\draw[-,red] (7,0)--(10,0);
\draw[-,blue](0,0)--(4,0);
\draw[blue] (4,0) to[out=0.40,in=180] (7,2);
\draw[-,blue] (7,2)--(10,2);
\draw[dashed] (4.01,0)--(4,4);
\draw[dashed] (7,0)--(7,4);
\draw[<->,purple!100](2,6)--(9,6);
\node[purple,below] at (5.5,7){\scalebox{0.7}{$\displaystyle{\sum_{j=1}^{\infty}\|\sqrt{d}\la v_j\|^2=o(\la^{-2})}$}};
\draw[dashed,purple] (2,0)--(2,6);
\draw[dashed,purple] (9,0)--(9,6);
\draw[<->,teal](3,5)--(8,5);
\node[teal,below] at (5.5,5.75) {\scalebox{0.5}{$\displaystyle{\sum_{j=1}^{\infty}\left(\|v_j'\|^2_{L^2(D_{\varepsilon})}+\xi_j^2\|v_j\|^2_{L^2(D_{\varepsilon})}\right),\sum_{j=1}^{\infty}\|\la p_j\|^2_{L^2(D_{\varepsilon})}=o(\la^{-2})}$}};
\draw[dashed,teal] (3,0)--(3,5);
\draw[dashed,teal] (8,0)--(8,5);
\draw[<->,violet](4,4)--(7,4);
\node[violet,below] at (5.5,4.75) {\scalebox{0.4}{$\displaystyle{\sum_{j=1}^{\infty}\left(\|p_j'\|^2_{L^2(D_{2\varepsilon})}+\xi_j^2\|p_j\|^2_{L^2(D_{2\varepsilon})}\right)=o(\la^{-2})}$}};
\draw[dashed,violet] (4,0)--(4,4);
\draw[dashed,violet] (7,0)--(7,4);
\node[black,right] at (12.5,4.5){\scalebox{0.75}{$h_1$}};
\node[black,right] at (12.5,4){\scalebox{0.75}{$h_2$}};
\draw[-,red](12,4.5)--(12.5,4.5);
\draw[-,blue](12,4)--(12.5,4);
\end{tikzpicture}
\caption{Geometric description of the functions $h_1$, $h_2$.}
\end{center}
\end{figure}

\begin{lemma}\label{INFO7-SQUARE}
The solution $(v,z,p,q)$ of system \eqref{1DETAIL} satisfies the following estimation
\begin{equation}\label{EQ1-INFO7-SQUARE}
\|U\|_{\mathcal{H}_{\square}}=o(1). 	
\end{equation}	
\end{lemma}
\begin{proof}
First, adding \eqref{EQ1-INFO5-SQUARE} and \eqref{EQ1-INFO6-SQUARE}, we get 
\begin{equation}\label{EQ2-INFO7-SQUARE}
2\alpha_1\sum_{j=1}^{\infty}\int_0^1h'\abs{v_j'}^2dx+2\beta \sum_{j=1}^{\infty}\int_0^1h'\abs{\gamma v_j'-p_j'}^2dx+2\sum_{j=1}^{\infty}\Re\left(i\la \int_0^1hdv_j\overline{v_j'}dx\right)+I=o(\la^{-2}).	
\end{equation}
Now, take $h(x)=xh_1+(x-1)h_2$. It is easy to see that 
\begin{equation}\label{h'h''}
h'(x)=xh_1'+h_1+(x-1)h_2'+h_2\quad \text{and}\quad h^{''}=xh_1^{''}+2h_1^{'}+(x-1)h_2^{''}+2h_2'. 	
\end{equation}
The aim now is to give an estimation for $I$. We start by using Cauchy-Schwarz inequality and \eqref{h'h''}, \eqref{EQ1-INFO1-SQUARE}, \eqref{EQ1-INFO2-SQUARE}, \eqref{EQ1-INFO3-SQUARE}, \eqref{EQ1-INFO4-SQUARE}, to get  
\begin{equation*}
\left\{\begin{array}{ll}
\displaystyle 
\left|\Re\left(\sum_{j=1}^{\infty}\int_0^1h^{''}v_j'\overline{v_j}dx\right)\right|=o(\la^{-3}),&\displaystyle 
\left|\Re\left(\sum_{j=1}^{\infty}\int_0^1h^{''}p_j'\overline{p_j}dx\right)\right|=o(\la^{-3})\\
\displaystyle 
\left|\Re\left(\sum_{j=1}^{\infty}\int_0^1h^{''}p_j'\overline{v_j}dx\right)\right|=o(\la^{-3}),&
\displaystyle
\left|\Re\left(\sum_{j=1}^{\infty}\int_0^1h^{''}v_j'\overline{p_j}dx\right)\right|=o(\la^{-3}).
\end{array}
\right.
\end{equation*}
Using the above estimations and \eqref{I}, we get 
\begin{equation}\label{EQ3-INFO7-SQUARE}
\abs{I}=o(\la^{-3}). 	
\end{equation}
Setting $\tilde{h}=xh_1'+(x-1)h_2'$, and using  \eqref{EQ1-INFO3-SQUARE}, \eqref{EQ1-INFO4-SQUARE}, we get 
$$
\sum_{j=1}^{\infty}\int_0^1\tilde{h}\abs{v_j'}^2dx=o(\la^{-2})\quad \text{and}\quad  \sum_{j=1}^{\infty}\int_0^1\tilde{h}\abs{\gamma v_j'-p_j'}^2dx=o(\la^{-2}).
$$
These estimations, \eqref{EQ3-INFO7-SQUARE}, \eqref{h'h''} and \eqref{EQ2-INFO7-SQUARE}, yield 
\begin{equation}\label{EQ4-INFO7-SQUARE}
\alpha_1\sum_{j=1}^{\infty}\int_0^1(h_1+h_2)\abs{v_j'}^2dx+\beta \sum_{j=1}^{\infty}\int_0^1(h_1+h_2)\abs{\gamma v_j'-p_j'}^2dx+\sum_{j=1}^{\infty}\Re\left(i\la \int_0^1hdv_j\overline{v_j'}dx\right)=o(\la^{-2}).	
\end{equation}
Using Young inequality and the definitions of the functions $d$ given by \eqref{dsquare}, $h$ given at the beginning of the proof of this Lemma and \eqref{EQ1-INFO1-SQUARE},  we get 
\begin{equation}\label{EQ5-INFO7-SQUARE}
\begin{array}{rll}
\displaystyle 
\left|\sum_{j=1}^{\infty}\Re\left(i\la \int_0^1hdv_j\overline{v_j'}dx\right)\right|&\leq &\displaystyle
\la\sum_{j=1}^{\infty}\int_0^1(h_1+h_2)d\abs{v_j}\abs{v_j'}dx\\
&\leq &\displaystyle 
\la \|\sqrt{d}\|_{\infty}\sum_{j=1}^{\infty}\int_0^1(h_1+h_2)\sqrt{d}
\abs{v_j}\abs{v_j'}dx\\
&\leq &\displaystyle 
\frac{\alpha_1}{2}\sum_{j=1}^{\infty}\int_0^1(h_1+h_2)\abs{v_j'}^2dx+\frac{ \|\sqrt{d}\|_{\infty}^2}{2\alpha_1}\sum_{j=1}^{\infty}\int_0^1(h_1+h_2)d\abs{\la v_j}^2dx\\
&\leq &\displaystyle 
\frac{\alpha_1}{2}\sum_{j=1}^{\infty}\int_0^1(h_1+h_2)\abs{v_j'}^2dx+\frac{ \|\sqrt{d}\|_{\infty}^2}{\alpha_1}\sum_{j=1}^{\infty}\int_0^1d\abs{\la v_j}^2dx\\
&\leq &\displaystyle 
\frac{\alpha_1}{2}\sum_{j=1}^{\infty}\int_0^1(h_1+h_2)\abs{v_j'}^2dx+o(\la^{-2}).
\end{array}	
\end{equation}
Inserting \eqref{EQ5-INFO7-SQUARE} in \eqref{EQ4-INFO7-SQUARE}, we get 
\begin{equation}\label{EQ6-INFO7-SQUARE}
\frac{\alpha_1}{2}\sum_{j=1}^{\infty}\int_0^1(h_1+h_2)\abs{v_j'}^2dx+\beta \sum_{j=1}^{\infty}\int_0^1(h_1+h_2)\abs{\gamma v_j'-p_j'}^2dx\leq o(\la^{-2}).	
\end{equation}
Using the definition of functions $h_1$ and $h_2$ and \eqref{EQ1-INFO3-SQUARE}, \eqref{EQ1-INFO4-SQUARE}, we get 
\begin{equation}\label{EQ6-INFO7-SQUARE}
\sum_{j=1}^{\infty}\int_0^1\abs{v_j'}^2dx=o(\la^{-2})\quad \text{and}\quad \sum_{j=1}^{\infty}\int_0^1\abs{\gamma v_j'-p_j'}^2dx=o(\la^{-2}). 
\end{equation}
From \eqref{EQ6-INFO7-SQUARE}, we get 
\begin{equation}\label{EQ7-INFO7-SQUARE}
\sum_{j=1}^{\infty}\int_0^1\abs{p_j'}^2dx\leq 2\sum_{j=1}^{\infty}\int_0^1\abs{\gamma v_j'-p_j'}^2dx+2\gamma^2\sum_{j=1}^{\infty}\int_0^1\abs{v_j'}^2dx\leq o(\la^{-2}).
\end{equation}
Using Poincar\'e inequality,  \eqref{EQ6-INFO7-SQUARE} and \eqref{EQ7-INFO7-SQUARE}, we get 
\begin{equation}\label{EQ8-INFO7-SQUARE}
\sum_{j=1}^{\infty}\int_0^1\abs{\la v_j}^2dx=o(1)\quad \text{and}\quad \sum_{j=1}^{\infty}\int_0^1\abs{\la p_j}^2dx=o(1). 	
\end{equation}
Using \eqref{EQ1-INFO6-SQUARE}, \eqref{EQ3-INFO7-SQUARE}, \eqref{EQ8-INFO7-SQUARE}, \eqref{EQ6-INFO7-SQUARE}, \eqref{EQ1-INFO3-SQUARE} and \eqref{EQ1-INFO4-SQUARE}, we get 
\begin{equation}\label{EQ9-INFO7-SQUARE}
\sum_{j=1}^{\infty}\int_0^1\xi_j^2\abs{v_j'}^2dx=o(1)\quad\text{and}\quad\sum_{j=1}^{\infty}\int_0^1\xi_j^2\abs{\gamma v_j-p_j}^2dx=o(1).	
\end{equation}
Finally, from \eqref{EQ6-INFO7-SQUARE}-\eqref{EQ9-INFO7-SQUARE}, we obtain \eqref{EQ1-INFO7-SQUARE}. The proof has been completed. 
\end{proof}

$\newline$
\noindent \textbf{Proof of Theorem \ref{Thm-Pol}.} From Lemma \ref{INFO7-SQUARE}, we have $\|U\|_{\mathcal{H}_{\square}}=o(1)$, which contradicts $\|U\|_{\mathcal{H}_{\square}}=1$ in \eqref{Cond-Pol1}. This implies that 
$$
\limsup_{\la\in \R,\ |\la|\to \infty}\frac{1}{\abs{\la}^{2}}\|(i\la I-A)^{-1}\|_{\mathcal{L}(H)}<\infty. 
$$
Finally, according to Theorem \ref{GENN}, we obtain the desired result. The proof has been completed.  

\section{Conclusion and open problems}
\noindent In this work, the local stabilization of 2D Pieozoelectric beam with magnetic effect on a rectangular domain without geometric conditions is considered.  The localized damping  regions do not satisfy the geometric control condition (GCC). Based on the frequency domain approach with the orthonormal basis decomposition and specific multiplier techniques, we have proved a polynomial energy decay rate of order $t^{-1}$. The  open problems that will be posed in this paper are:
\\
\begin{enumerate}
\item[${\rm (OP1)}$] What happens about the optimality of the polynomial energy decay rate?\\ 
\item[${\rm (OP2)}$] The case where the  damping region does not hit the boundary is still an open problem (See Figure \ref{Fig5}).
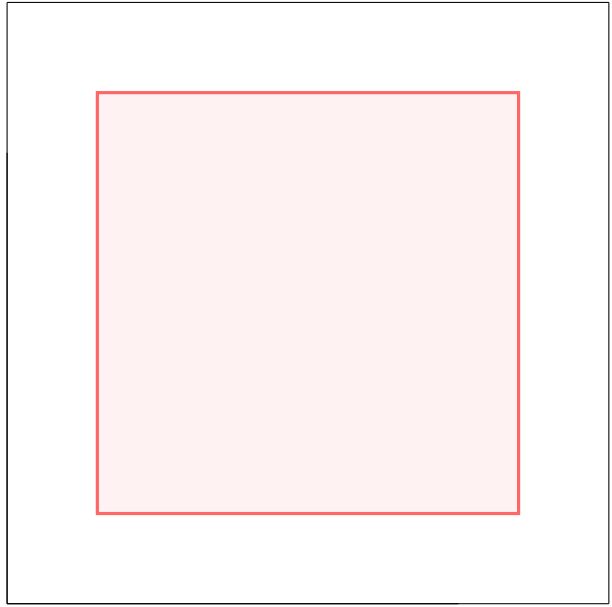
\begin{figure}[h!]
	\begin{center}
		\begin{tikzpicture}
			\draw[-](0,0)--(6,0);
			\draw[-](0,0)--(0,6);
			
			\draw[-,black](0,0)--(0,8);		
			\draw[-,black](0,0)--(8,0);
			\draw[-,black](8,0)--(8,8);
			\draw[-,black](0,8)--(8,8);				
			\draw[-,red](1.2,1.2)--(1.2,6.8);		
			\draw[-,red](1.2,6.8)--(6.8,6.8);	
			\draw[-,red](6.8,6.8)--(6.8,1.2);		
			\draw[-,red](6.8,1.2)--(1.2,1.2);
			\filldraw[color=red!60, fill=red!5, very thick] (1.2,1.2) rectangle (6.8,6.8);	
		\end{tikzpicture}
	\end{center}
	\caption{{\color{black}The case where the damping region (red part) is far away from the boundary.}}\label{Fig5}
\end{figure}

\item[${\rm (OP3)}$] The case where $\Omega$ is an open bounded regular domain (for example of class $C^2$) and the localised damping region does not satisfy the geometric control condition  (GCC) condition is still an open problem. For example, (See Figure \ref{ANNULUS}) 
$$
\left\{\begin{array}{l}
\displaystyle 
\Omega=\left\{X=(x,y)\in \mathbb{R}^2;\quad  r_1<\|X\|_{Euc}=\sqrt{x^2+y^2}<r_4\right\},\\[0.15in]
\displaystyle 
\Gamma_{0}=\left\{X=(x,y)\in \mathbb{R}^2;\quad  \|X\|_{Euc}=r_1\right\},\\[0.15in]
\displaystyle 
\Gamma_{1}=\left\{X=(x,y)\in \mathbb{R}^2;\quad  \|X\|_{Euc}=r_4\right\}.	
\end{array}
\right.
$$
and $d$ be a radial function, defined by 
\begin{equation*}
d(r)\geq d_1>0\ \text{in}\ \omega\quad \text{and}\quad d(r)=0\ \text{in}\ \Omega\backslash \omega, 
\end{equation*}
where $\omega:=\left\{X=(x,y)\in \mathbb{R}^2;\quad r_2<\|X\|_{Euc}=\sqrt{x^2+y^2}<r_3\right\}$ and $0<r_1<r_2<r_3<r_4$. \\
\item[${\rm (OP4)}$] The case where $\Omega$ is an open bounded regular domain (for example of class $C^2$) and the localised damping region satisfies the geometric control condition (GCC) condition is still an open problem. 
\end{enumerate}
\newpage

\begin{figure}[h!]
\begin{tikzpicture}
\draw[color=blue!60, fill=blue!5, very thick,dashed] (0,0) circle (4cm);
\draw[color=red!60, fill=red!5, very thick,dashed] (0,0) circle (3cm);
\draw[color=red!60, fill=blue!5, very thick,dashed] (0,0) circle (2cm);
\draw[color=purple!60, fill=white!5, very thick,dashed] (0,0) circle (1cm); 
\node[blue!60,left] at (4.2,0){\scalebox{1}{$\bullet$}};
\node[black] at (4.2,-0.15) {\scalebox{1}{$r_4$}};
\node[red!60,left] at (3.1999,0){\scalebox{1}{$\bullet$}};
\node[black] at (3.2,-0.15) {\scalebox{1}{$r_3$}};
\node[red!60,left] at (2.1999,0){\scalebox{1}{$\bullet$}};
\node[black] at (2.2,-0.15) {\scalebox{1}{$r_2$}};
\node[purple!60,left] at (1.1999,0){\scalebox{1}{$\bullet$}};
\node[black] at (1.2,-0.15) {\scalebox{1}{$r_1$}};
\node[black!60,left] at (2,1.5){\scalebox{1}{$\omega$}};
\node[blue] at (3,3.25){\scalebox{1}{$\Gamma_{1}$}};
\node[purple] at (0.5,0.5){\scalebox{1}{$\Gamma_{0}$}};
\node[blue!60,left] at (-3.8,0){\scalebox{1}{$\bullet$}};
\node[black] at (-4.3,-0.15) {\scalebox{1}{$-r_4$}};
\node[red!60,left] at (-2.8,0){\scalebox{1}{$\bullet$}};
\node[black] at (-3.3,-0.15) {\scalebox{1}{$-r_3$}};
\node[red!60,left] at (-1.8,0){\scalebox{1}{$\bullet$}};
\node[black] at (-2.3,-0.15) {\scalebox{1}{$-r_2$}};
\node[purple!60,left] at (-0.8,0){\scalebox{1}{$\bullet$}};
\node[black] at (-1.3,-0.15) {\scalebox{1}{$-r_1$}};
\node[blue!60,left] at (0.2,4){\scalebox{1}{$\bullet$}};
\node[black] at (0.15,4.2) {\scalebox{1}{$r_4$}};
\node[red!60,left] at (0.2,3){\scalebox{1}{$\bullet$}};
\node[black] at (0.2,3.2) {\scalebox{1}{$r_3$}};
\node[red!60,left] at (0.2,2){\scalebox{1}{$\bullet$}};
\node[black] at (0.2,2.2) {\scalebox{1}{$r_2$}};
\node[purple!60,left] at (0.2,1){\scalebox{1}{$\bullet$}};
\node[black] at (0.2,1.2) {\scalebox{1}{$r_1$}};
\node[blue!60,left] at (0.2,-4){\scalebox{1}{$\bullet$}};
\node[black] at (0.35,-4.2) {\scalebox{1}{$-r_4$}};
\node[red!60,left] at (0.2,-3){\scalebox{1}{$\bullet$}};
\node[black] at (0.4,-3.2) {\scalebox{1}{$-r_3$}};
\node[red!60,left] at (0.2,-2){\scalebox{1}{$\bullet$}};
\node[black] at (0.4,-2.2) {\scalebox{1}{$-r_2$}};
\node[purple!60,left] at (0.2,-1){\scalebox{1}{$\bullet$}};
\node[black] at (0.4,-1.2) {\scalebox{1}{$-r_1$}};
\draw[<->] (0.75,-0.75) -- (2.75,-2.95);
\node[black] at (2.75,-2.5) {\scalebox{1}{$\Omega$}};
\node[black] at (-2,1.5) {\scalebox{0.75}{$d(X)>0$}};
\node[black] at (-3.17,1.5) {\scalebox{0.75}{$d(X)=0$}};
\node[black] at (-0.75,1.25) {\scalebox{0.75}{$d(X)=0$}};
\draw[<->] (1.4,1.4)--(2.,2.25);
\draw[arrows=->] (1,0) -- (5,0); 
\draw[arrows=->] (0,1) -- (0,5);
\draw[arrows=->] (-1,0) -- (-5,0);
\draw[arrows=->] (0,-1) -- (0,-5);
\end{tikzpicture}
\caption{Model Describing $\Omega$, $\Gamma_{0}$, $\Gamma_{1}$ and $\omega$ in (OP3).}\label{ANNULUS}
\end{figure}
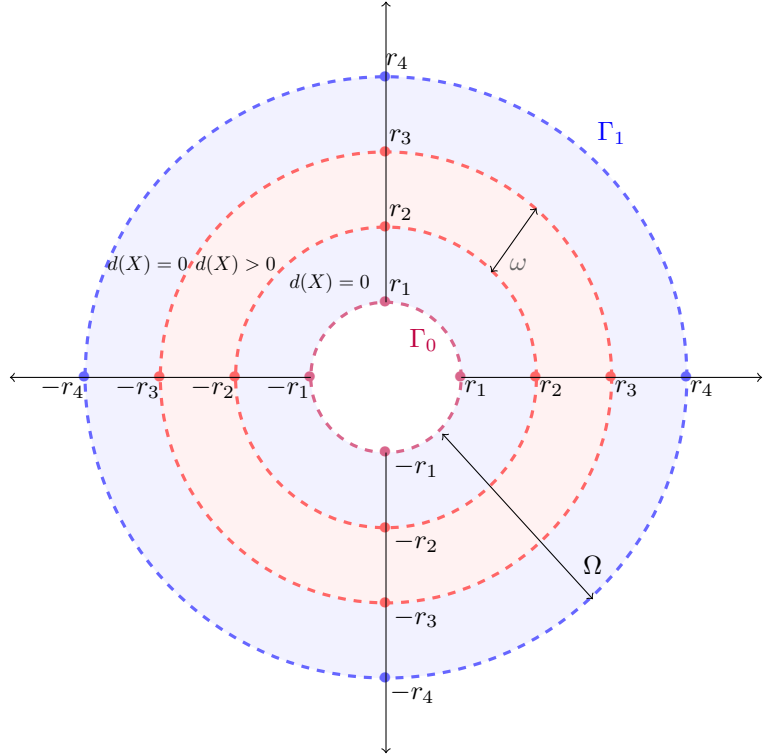


\begin{thebibliography}{10}

\bibitem{Abdelaziz1}
M.~Afilal, A.~Soufyane, and M.~de~Lima~Santos.
\newblock Piezoelectric beams with magnetic effect and localized damping.
\newblock {\em Mathematical Control \& Related Fields}, 0:--, 2021.

\bibitem{https://doi.org/10.48550/arxiv.2204.00283}
M.~Akil.
\newblock Stability of pizoelectric beam with magnetic effect under (coleman or
  pipkin)-gurtin thermal law, 2022.

\bibitem{https://doi.org/10.48550/arxiv.2111.14554}
M.~Akil, H.~Badawi, S.~Nicaise, and V.~Régnier.
\newblock Stabilization of coupled wave equations with viscous damping on
  cylindrical and non-regular domains: Cases without the geometric control
  condition, 2021.

\bibitem{Akil2022}
M.~Akil, I.~Issa, and A.~Wehbe.
\newblock A n-dimensional elastic $\backslash$ viscoelastic transmission
  problem with Kelvin-Voigt damping and non smooth coefficient at the
  interface.
\newblock {\em SeMA Journal}, May 2022.

\bibitem{AnLiuKong}
Y.~An, W.~Liu, and A.~Kong.
\newblock Stability of piezoelectric beams with magnetic effects of fractional
  derivative type and with/without thermal effects.
\newblock 2021.

\bibitem{BPS}
C.~Batty, L.~Paunonen, and D.~Seifert.
\newblock Optimal energy decay for the wave-heat system on a rectangular
  domain.
\newblock {\em SIAM Journal on Mathematical Analysis}, 51(2):808--819, 2019.

\bibitem{Batty01}
C.~J.~K. Batty and T.~Duyckaerts.
\newblock Non-uniform stability for bounded semi-groups on {B}anach spaces.
\newblock {\em J. Evol. Equ.}, 8(4):765--780, 2008.

\bibitem{Borichev01}
A.~Borichev and Y.~Tomilov.
\newblock Optimal polynomial decay of functions and operator semigroups.
\newblock {\em Math. Ann.}, 347(2):455--478, 2010.

\bibitem{Hayek}
A.~Hayek, S.~Nicaise, Z.~Salloum, and A.~Wehbe.
\newblock A transmission problem of a system of weakly coupled wave equations
  with {K}elvin-{V}oigt dampings and non-smooth coefficient at the interface.
\newblock {\em SeMA J.}, 77(3):305--338, 2020.

\bibitem{RaoLiu01}
Z.~Liu and B.~Rao.
\newblock Characterization of polynomial decay rate for the solution of linear
  evolution equation.
\newblock {\em Z. Angew. Math. Phys.}, 56(4):630--644, 2005.

\bibitem{Morris-Ozer2013}
K.~Morris and A.~A. {\"O}zer.
\newblock Strong stabilization of piezoelectric beams with magnetic effects.
\newblock pages 3014--3019, 2013.

\bibitem{Morris-Ozer2014}
K.~A. Morris and A.~{\"O}. {\"O}zer.
\newblock Modeling and stabilizability of voltage-actuated piezoelectric beams
  with magnetic effects.
\newblock {\em SIAM J. Control. Optim.}, 52:2371--2398, 2014.

\bibitem{Ramos2018}
{Ramos, Anderson J.A.}, {Gon\c{c}alves, Cledson S.L.}, and {Corr\^ea Neto,
  Silv\'erio S.}
\newblock Exponential stability and numerical treatment for piezoelectric beams
  with magnetic effect.
\newblock {\em ESAIM: M2AN}, 52(1):255--274, 2018.

\bibitem{ROZENDAAL2019359}
J.~Rozendaal, D.~Seifert, and R.~Stahn.
\newblock Optimal rates of decay for operator semigroups on {H}ilbert spaces.
\newblock {\em Advances in Mathematics}, 346:359 -- 388, 2019.

\bibitem{Abdelaziz2}
A.~Soufyane, M.~Afilal, and M.~L. Santos.
\newblock Energy decay for a weakly nonlinear damped piezoelectric beams with
  magnetic effects and a nonlinear delay term.
\newblock {\em Zeitschrift f\"{u}r angewandte Mathematik und Physik}, 72(4),
  Aug. 2021.

\bibitem{Stahn2017}
R.~Stahn.
\newblock Optimal decay rate for the wave equation on a square with constant
  damping on a strip.
\newblock {\em Zeitschrift f\"{u}r angewandte Mathematik und Physik}, 68(2),
  Feb. 2017.

\bibitem{doi:10.1137/20M1332499}
K.~Yu and Z.-J. Han.
\newblock Stabilization of wave equation on cuboidal domain via
  {K}elvin--{V}oigt damping: A case without geometric control condition.
\newblock {\em SIAM Journal on Control and Optimization}, 59(3):1973--1988,
  2021.

\bibitem{Zhang2022}
H.-E. Zhang, G.-Q. Xu, and Z.-J. Han.
\newblock Stability of multi-dimensional nonlinear piezoelectric beam with
  viscoelastic infinite memory.
\newblock {\em Zeitschrift f{\~A}r angewandte Mathematik und Physik},
  73(4):159, Jul 2022.

\end{thebibliography}

\end{document}